\documentclass[11pt]{article}
\usepackage{import}
\usepackage[utf8]{inputenc}
\usepackage{latexsym}
\usepackage{amsmath,amsxtra,amssymb,latexsym, amscd,amsthm,}
 \usepackage[pdftex]{color,graphicx}
 \usepackage[all]{xy}
 \usepackage{hyperref}
\usepackage{pgf,tikz}
\usepackage[T1]{fontenc}
\usepackage{cleveref}
\usepackage{amsfonts}
\usetikzlibrary{arrows}
\usepackage{graphicx}
\graphicspath{{figs/}}
\usepackage{caption}
\usepackage{subcaption}
\usepackage{graphpap}
\usepackage{rotating}
\usepackage{bm}

\theoremstyle{definition}

\theoremstyle{plain}
\newtheorem{cor}{Corollary}
\newtheorem{thm}{Theorem}

\newtheorem{lem}{Lemma}
\newtheorem{defn}{Definition}

\newtheorem{prop}{Proposition}
\usepackage{fancyhdr}
\newcommand\shorttitle{Diverging on average set}
\newcommand\authors{C. Lo, S. Herrero Vila}
\theoremstyle{definition}

\newtheorem{rem}{Remark}

\newcommand{\R}{\mathbb{R}}

\title{\sc{Geodesic rays diverging on average on a pair of pants}}
\author{Sergio Herrero Vila and Cheikh Lo}
\date{ \small{}}

\begin{document}
\renewcommand{\proofname}{Proof}
\renewcommand{\abstractname}{Abstract}
\renewcommand{\refname}{Bibliography}
\maketitle
\begin{abstract} The aim of this paper is to characterize in terms of coding a set of limit points considered in \cite{felipe} by F. Riquelme and A. Velozo corresponding to geodesic rays which spend less time in any compact region of a pair of pants with one cusp. Moreover in this particular context we reprove that the Hausdorff dimension of this set is equal $ \frac{1}{2} $ by explicit calculus.
\end{abstract}
\textbf{Key  words :} geodesic rays, radial limit points, divergence on average, Hausdorff dimension.\\
AMS 2010 \textit{Mathematics Subject Classification.}\\
\textbf{Acknowledgment :} The authors would like to express their sincere thanks to Cirm and Aims-Mbour. They benefited from the Cirm-Aims Research in Residence project. The first author would also like to thank Irmar-University of Rennes 1 for its welcome and hospitality during its research stay in November 2023.

\section{Introduction}
Let $ \Sigma $ be a hyperbolic pair of pants with one cusp. By the uniformization theorem, such a surface is the quotient $ \Gamma\backslash \mathbb{D} $, where $ \mathbb{D} $ is the Poincare disc model and $ \Gamma $ is a discrete subgroup of $ Isom^{+}(\mathbb{D}) $ generated by a hyperbolic and a parabolic isometry. We compactify the model $ \mathbb{D} $ by adding a boundary at infinity $ \partial\mathbb{D}=\mathbb{S}^{1} $. The \textit{limit set} of a Fuchsian group $ \Gamma $ of $ \mathbb{D} $, denoted $ \Lambda $ is the subset of $ \partial\mathbb{D} $ defined for any $ z\in\mathbb{D} $ by $ \Lambda =\overline{\Gamma.z}\cap\partial\mathbb{D}$. The set $ \Lambda $ is the smallest closed $ \Gamma $-invariant subset of $ \partial\mathbb{D}=\mathbb{S}^{1} $.\\
This paper revolves around the study of points at infinity that are endpoints of geodesic rays that spend less time in any compact region of a pair of pants with one cusp. We evaluate the average time spent by a geodesic ray within a compact subset of the surface $\Sigma$. If this rate goes to zero, the endpoint $ \tilde{\sigma}(\infty)\in\partial\mathbb{D} $ of a lift $\tilde{\sigma}$ of $ \sigma $ is called a \textit{limit point diverging on average}. Formally
\begin{defn}\label{defpointaverage}
An element $ \xi $ of the boundary at infinity $ \partial\mathbb{D}=\mathbb{S}^{1} $ is said to be a \textbf{diverging on average limit point} if it is a limit point and for any geodesic ray $ \sigma:[0,\infty)\rightarrow\Sigma $ such that the endpoint of one (and thus every) of its lifts $ \tilde{\sigma}(\infty)=\xi $ and for any compact set $ W $ in $ \Sigma $ we have
	\begin{equation}\label{eq_divergencia}
	\lim_{T\rightarrow+\infty}\frac{1}{T}\int_{0}^{T}\chi_{W}(\sigma(t))dt=0.
	\end{equation}
	where $ \chi_{W} $ is the indicator function of $ W $. The set of diverging on average limit points is called \textbf{diverging on average set} and it is denoted $ \Lambda_{\infty} $.
\end{defn}

The study of such a set of points at infinity or their corresponding geodesic is part of a long tradition of works. Very early planetary motion interested the pioneers of dynamical systems theory. They studied the asymptotic behavior of special trajectories and related them to number theory. For more information about this tradition of study we recommend \cite{dani} and \cite{weiss}.\\
This work starts from a one of F. Riquelme and A. Velozo, \cite{felipe}, and was intended to be an explicit example of the latter. Naturally we started by adopting their techniques in our specific case. We quickly realized that it would be better by computing explicitly the Hausdorff dimension of the set of average limit points. This was possible thanks to a coding. of this set established in the earlier section. This manuscript is designed to be as self-contained as possible, offering an explicit characterization of the encoding of the divergent on average limit points of a Schottky group generated by a hyperbolic and a parabolic isometry. In that we set out the following.
\begin{thm}\label{average}
Let $ \Gamma $ be a Schottky group generated by a hyperbolic isometry $ h $ and a parabolic isometry $ p $. Then $ \xi\in\Lambda_{\infty} $ if only if for any positive integer $ N $ there exists a coding in terms $ \omega_{i} $ and $ p^{r_{i}} $ such that
$$ \lim_{q\rightarrow+\infty}\frac{\sum_{i=0}^{q}d(0,\omega_{i}(0))}{\sum_{i=1}^{q}2\ln|r_{i}|}=0 $$ 
where $ d $ is the hyperbolic distance on $ \mathbb{D} $ and for every $ i\in[1,q]\;\;\mbox{and}\;\;k\in [1,s_{i}]$ 
$$|r_{i}|\geq N\;\;\mbox{and}\;\;\omega_{i}=h^{j_{1}}p^{l_{1}}h^{j_{2}}p^{l_{2}}...
p^{l_{s_{i}-1}}h^{j_{s_{i}}}\;\;\mbox{with}\;\;|l_{k}|<N\;\;\mbox{and}\;\;
j_{k}\in\mathbb{Z}^{\ast}. $$
\end{thm}

Moreover, we provide a detailed computation of the Hausdorff dimension of this subset of boundary points, illustrating that it exemplifies a subset of the circumference with a Hausdorff dimension of $\frac{1}{2}$.

\begin{thm}\label{hausd}
The Hausdorff dimension of the set of limit points diverging on average is equal to the critical exponent of the parabolic subgroup of $ \Gamma $. Precisely
	$$ dim_{H}(\Lambda_{\infty})=\frac{1}{2}.$$
\end{thm}

The following is organized into three sections. The first one contains the background and statements of the results of this paper. The second section aims to study the diverging on average set and establishes a condition in terms of coding. The last section gives an interpretation in terms of measure theory. Precisely we calculate the Hausdorff dimension of the diverging on average set.
\section{Background and statements of results}

A geodesic of $ \mathbb{D} $ is either a diameter of $ \mathbb{D} $ or semicircle orthogonal to unit circle $ \mathbb{S}^{1} $. We call \textit{horocycle} of $ \mathbb{D} $ a Euclidean circle tangent to $ \mathbb{S}^{1} $ at a complex number $ z\in\mathbb{S}^{1} $. We write that curve $ \mathcal{H}_{z} $. Its corresponding disc is called \textit{horodisc centered at} $ z $ and denoted  $ \mathcal{HD}_{z} $.\\ 
Let $ g$ be an element of $ Isom^{+}(\mathbb{D}) $ and $ z_{0}\in\mathbb{D} $ such that $ g(z_{0})\neq z_{0} $. We denote by $ C(g) $ the bisector of the hyperbolic segment $ [z_{0},g(z_{0})] $ and $ D(g) $ the half-plane of $ \mathbb{D} $ delimited by $ C(g) $ containing $ g(z_{0}) $ and put 
$ E(g)=\mathbb{D}\setminus\mathring{D}(g) $. We have $ g(D(g^{-1}))=E(g) $. According to the relative position of $ C(g) $ and $ C(g^{-1}) $ we have the following classification :
\begin{itemize}
	\item $ g $ is called \textit{hyperbolic} if $ C(g)\cap C(g^{-1})=\emptyset $. In this case the only fixed points of $ g $ are the endpoints of the commune perpendicular to  $ C(g) $ and $ C(g^{-1}) $.
	\item $ g $ is called \textit{parabolic} if $ C(g)\cap C(g^{-1})=\{\xi\}\subset\partial\mathbb{D} $. In this case $ \xi $ is the unique fixed point of $ g $.
	\item $ g $ is called \textit{elliptic} if $ C(g)\cap C(g^{-1})=\{z\}\subset\mathbb{D} $. In this case $ z $ is the unique fixed point of $ g $ in $ \mathbb{D} $.
\end{itemize} 
As announced at the introduction of this paper, we consider a Fuchsian group $ \Gamma $ generated by a hyperbolic isometry denoted $ h $ and a parabolic isometry denoted $ p $ of $ \mathbb{D} $ such that 
$ \overline{D(h)\cup D(h^{-1})}\cap\overline{D(p)\cup D(p^{-1})}=\emptyset $ and $ \Delta=\bigcap_{g\in\mathcal{A}_{1}}E(g) $, where $ \mathcal{A}_{1}=\{h,h^{-1},p,p^{-1}\} $, is the Dirichlet domain centered at $ 0 $ and so a Ford domain, see figure \ref{dom}. 


\begin{figure}[htbp]
	\centering
	\begin{normalsize} 
		\def\svgwidth{\textwidth}
\begingroup%
  \makeatletter%
  \providecommand\color[2][]{%
    \errmessage{(Inkscape) Color is used for the text in Inkscape, but the package 'color.sty' is not loaded}%
    \renewcommand\color[2][]{}%
  }%
  \providecommand\transparent[1]{%
    \errmessage{(Inkscape) Transparency is used (non-zero) for the text in Inkscape, but the package 'transparent.sty' is not loaded}%
    \renewcommand\transparent[1]{}%
  }%
  \providecommand\rotatebox[2]{#2}%
  \newcommand*\fsize{\dimexpr\f@size pt\relax}%
  \newcommand*\lineheight[1]{\fontsize{\fsize}{#1\fsize}\selectfont}%
  \ifx\svgwidth\undefined%
    \setlength{\unitlength}{1995bp}%
    \ifx\svgscale\undefined%
      \relax%
    \else%
      \setlength{\unitlength}{\unitlength * \real{\svgscale}}%
    \fi%
  \else%
    \setlength{\unitlength}{\svgwidth}%
  \fi%
  \global\let\svgwidth\undefined%
  \global\let\svgscale\undefined%
  \makeatother%
  \begin{picture}(1,0.58395987)%
    \lineheight{1}%
    \setlength\tabcolsep{0pt}%
    \put(0,0){\includegraphics[width=\unitlength,page=1]{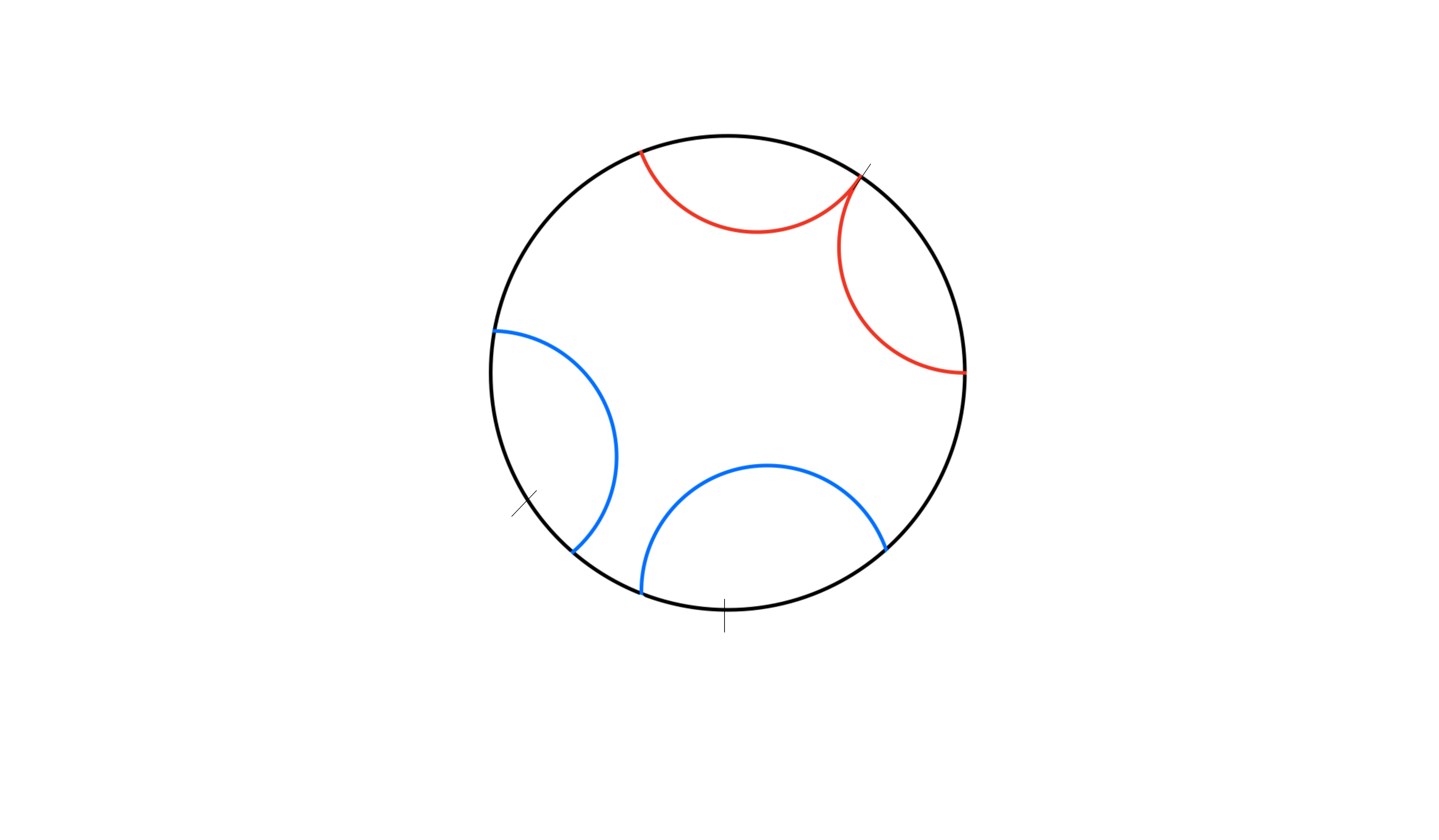}}%
    \put(0.61320755,0.47220369){\color[rgb]{0.05490196,0.05490196,0.05490196}\makebox(0,0)[lt]{\lineheight{1.25}\smash{\begin{tabular}[t]{l}$p^+$\end{tabular}}}}%
    \put(0.50290275,0.130404){\color[rgb]{0.05490196,0.05490196,0.05490196}\makebox(0,0)[lt]{\lineheight{1.25}\smash{\begin{tabular}[t]{l}$h^+$\end{tabular}}}}%
    \put(0.33309143,0.20007017){\color[rgb]{0.05490196,0.05490196,0.05490196}\makebox(0,0)[lt]{\lineheight{1.25}\smash{\begin{tabular}[t]{l}$h^-$\end{tabular}}}}%
  \end{picture}%
\endgroup%

	\end{normalsize}
	\caption{Dirichlet (Ford) domain of $ \Gamma $ centered at $ 0 $.}
	\label{dom}
\end{figure}
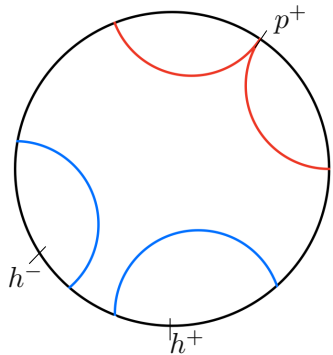


The group $ \Gamma=<h,p> $ is a \textit{Schottky group of rank two}. One can think $\Sigma=\Gamma\backslash\mathbb{D}$ as a compact surface $ \Sigma_{0} $ on which one \textit{cusp} and two \textit{funnels}. For convenience of computation often we resort to another model of the hyperbolic plane, the Poincare upper half-plane, $ \mathbb{H} $. This model will be related to the disc model by an isometry $ \varphi:\mathbb{D}\rightarrow\mathbb{H} $ such that the parabolic element $ p $ of $ \mathbb{D} $ is conjugate by $ \varphi $ to the element $ P:z\mapsto z+1 $ of $ PSL(2,\mathbb{R}) $, meaning $ P=\varphi\circ p\circ\varphi^{-1} $. A geodesic of $ \mathbb{H} $ is either an Euclidean semicircle centered on the real axis or a line perpendicular to the real axis. We call \textit{horocycle} of $ \mathbb{H} $ a Euclidean circle tangent at a real $ x $ to the real axis or a horizontal line of $ \mathbb{H} $. They are written respectively $ \mathcal{H}_{x} $ and $ \mathcal{H}_{\infty} $ and called \textit{horocycle centered at $ x $} or $ \infty $. A \textit{horodisc} is a Euclidean disc or a upper-region of $ \mathbb{H} $ delimited by a horocycle. They are respectively written $ \mathcal{HD}_{x} $ and $ \mathcal{HD}_{\infty} $.\\
The limit set $ \Lambda $ of $ \Gamma $ is divided into two disjoint subsets. A point $ \xi $ of $ \Lambda $ where there exists a sub-orbit of $ \Gamma.z $ that converges to $ \xi $ within a bounded hyperbolic distance from a geodesic ray ending at $ \xi $ is called a \textit{radial limit point} and the set of thus points is denoted $ \Lambda_{r} $. Every fix-point of a hyperbolic isometry is a radial limit point. The fix-point of a parabolic isometry is a limit point of $ \Gamma $ called \textit{parabolic limit point}. We denote by $ \Lambda_{p} $ the set of parabolic limit points of $ \Gamma $. Hence we have the disjoint union $ \Lambda=\Lambda_{r}\sqcup\Lambda_{p} $.
\begin{rem}
By the definitions of the sets $ \Lambda_{r} $ and $ \Lambda_{\infty} $ and the disjointedness $ \Lambda=\Lambda_{r}\sqcup\Lambda_{p} $ if $ \xi\in\Lambda_{\infty} $, so is $ \gamma.\xi $ for all $ \gamma\in\Gamma $ and clearly $ \Lambda_{p}\subset\Lambda_{\infty} $. Therefore, we are right to be interested by the following questions : Which radial limit points are in $ \Lambda_{\infty} $ ? Which is their characterization and encoding ? Is there an interpretation of the set $ \Lambda_{\infty}$ in terms of measure theory since it belongs to the limit set $\Lambda $?
\end{rem}

\section{The diverging on average set}
\subsection{Behavior of geodesic rays in the surface}
For a Fuchsian group $ \Gamma $ the convex hull of $ \Lambda $, written $ \widetilde{N}(\Lambda) $ and called often the \textit{Nielsen region}, is the minimal convex subset of $ \mathbb{D} $ that contains all geodesics with endpoints in $ \Lambda $. The \textit{convex core} $ N(\Sigma) $ of the associated hyperbolic surface $ \Sigma $ is the quotient $ \Gamma\backslash\widetilde{N}(\Lambda) $. The convex core is the smallest sub-surface of $ \Sigma $ such that the inclusion map is a homotopy equivalence. Every parabolic fix-point $ x $ corresponds on the hyperbolic surface $ \Sigma $ to a cusp, denoted $ \mathcal{C}_{x} $. In the case of this paper, the surface $\Sigma=\Gamma\backslash\mathbb{D}$, where $ \Gamma=<h,p> $, comprises two funnels, a unique cusp denoted $ \mathcal{C}_{p^{+}} $, where $p^{+} $ is the fix-point of the parabolic element $ p $ and the compact core $ \Sigma_{0}=N(\Sigma)\setminus\mathcal{C}_{p^{+}} $.
The cusp $ \mathcal{C}_{p^{+}} $ has a cusp neighborhood filled by embedded horocycles winding around it. The fix-point $p^{+}\in\mathbb{S}^{1} $ corresponds, in the upper-half plane model $ \mathbb{H} $, to the point $ \infty $. We call a \textit{a cusp of level $ k\in\mathbb{N}^{\ast} $}, denoted $ \mathcal{C}_{\infty}(k) $, the projection $ \pi:\mathbb{H}\rightarrow\Sigma=\Gamma\backslash\mathbb{H} $ of the horodisc $\mathcal{HD}_{\infty}(k)=\{z\in\mathbb{H}: Im(z)>k \} $. In the case of the disc model, a cusp of level $ k $ is the set $ \mathcal{C}_{p^{+}}(k)=\pi\circ\varphi^{-1}(\mathcal{HD}_{\infty}(k)) $. For every compact $ W \subset N(\Sigma)$ there exists a natural integer $ k $ such that $ W\subset N(\Sigma)\setminus\mathcal{C}_{p^{+}}(k) $, see Figure \ref{dibujo12}

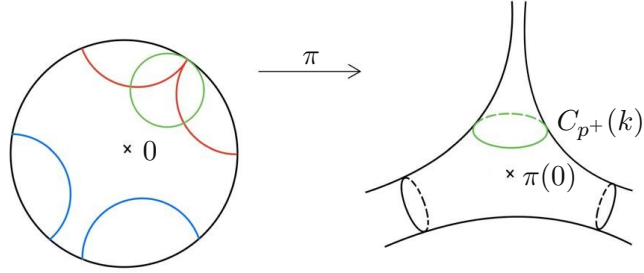
\begin{figure}[htbp]
	\centering
	\begin{normalsize} 
		\def\svgwidth{\textwidth}
\begingroup%
  \makeatletter%
  \providecommand\color[2][]{%
    \errmessage{(Inkscape) Color is used for the text in Inkscape, but the package 'color.sty' is not loaded}%
    \renewcommand\color[2][]{}%
  }%
  \providecommand\transparent[1]{%
    \errmessage{(Inkscape) Transparency is used (non-zero) for the text in Inkscape, but the package 'transparent.sty' is not loaded}%
    \renewcommand\transparent[1]{}%
  }%
  \providecommand\rotatebox[2]{#2}%
  \newcommand*\fsize{\dimexpr\f@size pt\relax}%
  \newcommand*\lineheight[1]{\fontsize{\fsize}{#1\fsize}\selectfont}%
  \ifx\svgwidth\undefined%
    \setlength{\unitlength}{957bp}%
    \ifx\svgscale\undefined%
      \relax%
    \else%
      \setlength{\unitlength}{\unitlength * \real{\svgscale}}%
    \fi%
  \else%
    \setlength{\unitlength}{\svgwidth}%
  \fi%
  \global\let\svgwidth\undefined%
  \global\let\svgscale\undefined%
  \makeatother%
  \begin{picture}(1,0.60423197)%
    \lineheight{1}%
    \setlength\tabcolsep{0pt}%
    \put(0,0){\includegraphics[width=\unitlength,page=1]{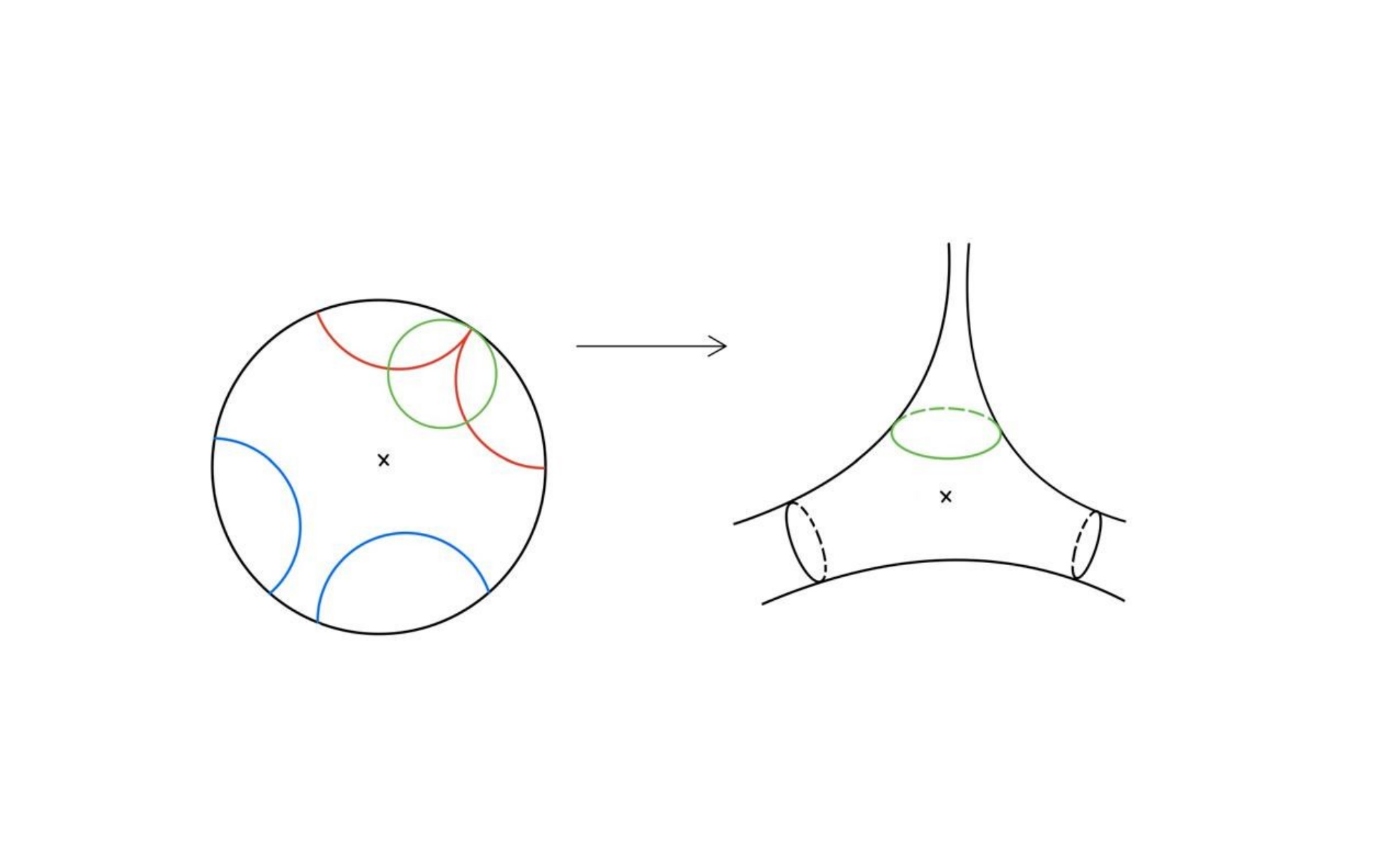}}%
    \put(0.28997248,0.26347753){\color[rgb]{0.05490196,0.05490196,0.05490196}\makebox(0,0)[lt]{\lineheight{1.25}\smash{\begin{tabular}[t]{l}$0$\end{tabular}}}}%
    \put(0.6888686,0.23771854){\color[rgb]{0.05490196,0.05490196,0.05490196}\makebox(0,0)[lt]{\lineheight{1.25}\smash{\begin{tabular}[t]{l}$\pi(0)$\end{tabular}}}}%
    \put(0.7249312,0.29512426){\color[rgb]{0.05490196,0.05490196,0.05490196}\makebox(0,0)[lt]{\lineheight{1.25}\smash{\begin{tabular}[t]{l}$C_{p^+}(k)$\end{tabular}}}}%
    \put(0.45777381,0.36504149){\color[rgb]{0.05490196,0.05490196,0.05490196}\makebox(0,0)[lt]{\lineheight{1.25}\smash{\begin{tabular}[t]{l}$\pi$\end{tabular}}}}%
  \end{picture}%
\endgroup%

	\end{normalsize}
	\caption{The surface $\Sigma_0$.}
	\label{dibujo12}
\end{figure}

\begin{rem}
	Note that it is not restrictive to assume the compact subsets of $ \Sigma $ in the definition \ref{defpointaverage} to be inside the Nielsen region, as, in order to verify if an endpoint of a ray $(\sigma(t))$ is diverging on average, we just to check equation (\ref{eq_divergencia}) for compact sets $W\subset N(\Sigma)$, because, for any compact set $W\subset\Sigma,$ we can see that
	\begin{align*}
		&\lim_{T\rightarrow+\infty}\frac{1}{T}\int_{0}^{T}\chi_{W}(\sigma(t))dt
		=\lim_{T\rightarrow+\infty}\frac{1}{T}\int_{0}^{T}(\chi_{W\cap N(\Sigma)}(\sigma(t))+\chi_{W\cap (S\backslash N(\Sigma))}(\sigma(t)))dt\\
		&=\lim_{T\rightarrow+\infty}\frac{1}{T}\int_{0}^{T}\chi_{W\cap N(\Sigma)}(\sigma(t))dt+ 
		\lim_{T\rightarrow+\infty}\frac{1}{T}\int_{0}^{T}\chi_{W\cap\Sigma\backslash N(\Sigma)}(\sigma(t))dt\\
		&=\lim_{T\rightarrow+\infty}\frac{1}{T}\int_{0}^{T}\chi_{W\cap N(\Sigma)}(\sigma(t))dt,
	\end{align*}
	where the last equality follows from the fact that if a ray exits the Nielsen region $N(\Sigma)$, it is going to diverge to infinity without coming back to $N(\Sigma)$.  	
\end{rem}

\begin{prop}
Let $ \tilde{\sigma} $ be a lift of a geodesic ray $ \sigma $ in the surface $ \Sigma $. If $ \tilde{\sigma}(\infty)\notin\Gamma.p^{+} $ then there is an unbounded sequence of positive reals $ (t_{n}) $ such that $ \sigma(t_{n})\in B(\pi(0),R) $, where $ B(\pi(0),R) $ is a fixed ball of $N(\Sigma)\setminus\mathcal{C}_{p^{+}} $.
\end{prop}

\begin{proof}
	The geodesic $\tilde\sigma(t)$ is directed to a limit point which is not parabolic, as $\tilde\sigma(\infty)\notin \Gamma.p^+$. We then deduce that $\tilde\sigma(\infty)$ is radial limit point, which means that there exist a compact $K\subset\Sigma$ and an unbounded sequence $(t_n)$ such that  $\sigma(t_n)\in K$ for all $n$. We now consider a ball in the surface centered in $\pi(0)$ with a radius $R$ big enough so that $B(\pi(0),R)\supset K$.
	
\end{proof}

Consequently for any geodesic ray $ \sigma $ such that $ \tilde{\sigma}(\infty)\in\Lambda_{r} $ there exists $ R>0 $ such that $ \sigma $ returns infinitely many times in the ball $ B(\pi(0),R) $. A geodesic ray $ \sigma $ such that $ \tilde{\sigma}(\infty)\in\Lambda_{p} $ exits any compact region of $ \Sigma $ and enters in the cusp without backtrack.

\subsection{Cusp crossing time}\label{cuscros}
Let $ \sigma=\pi([0,\xi)) $, where $ \xi\in\Lambda $, be a geodesic ray which crosses a cusp.
In the half-plane Poincare model, we compute the length of $ \sigma $ in  $ \mathcal{C}_{\infty}(k) $. For every $ z\in\mathcal{H}_{\infty}(k) $, $$ d(z,p^n(z))=2\log{\left(\frac{n}{2Imz}+ \sqrt{\frac{n^2}{4Im^{2}z}-1}\right)}\sim 2\log\Big(\frac{n}{k}\Big), $$ where $ \sim $ means the two functions are equivalent. This distance exists iff $ n\geqslant 2k $. If it is the case we say that $ \sigma $ \textit{winds around the cusp} $ \mathcal{C}_{\infty}(k) $ and the power $ n $ of the parabolic isometry $ p $ is called \textit{a winding number} and denoted $ \omega(\sigma,k) $ or simply $ \omega(\sigma) $ if there is not confusion. An $ (n,k) $-excursion is a geodesic arc $ \sigma' $ of a geodesic ray $ \sigma $ lying on $ \mathcal{C}_{\infty}(k) $ with endpoints on the boundary horocycle $ \pi(\mathcal{H}_{\infty}(k)) $ and with winding number $ n $ around the cusp $ \mathcal{C}_{\infty}(k) $. Someone is interested in the sum of the lengths of all excursions along the geodesic ray up to a given time. One has the following and for a proof we refer to \cite{basm}-Lemma 4.1 or \cite{Vo}-Corollary 2.2.

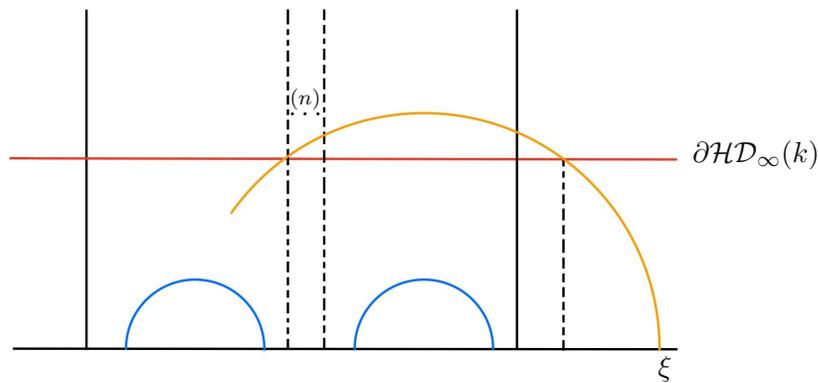
\begin{figure}[htbp]
	\centering
	\begin{normalsize} 
		\def\svgwidth{\textwidth}
\begingroup%
  \makeatletter%
  \providecommand\color[2][]{%
    \errmessage{(Inkscape) Color is used for the text in Inkscape, but the package 'color.sty' is not loaded}%
    \renewcommand\color[2][]{}%
  }%
  \providecommand\transparent[1]{%
    \errmessage{(Inkscape) Transparency is used (non-zero) for the text in Inkscape, but the package 'transparent.sty' is not loaded}%
    \renewcommand\transparent[1]{}%
  }%
  \providecommand\rotatebox[2]{#2}%
  \newcommand*\fsize{\dimexpr\f@size pt\relax}%
  \newcommand*\lineheight[1]{\fontsize{\fsize}{#1\fsize}\selectfont}%
  \ifx\svgwidth\undefined%
    \setlength{\unitlength}{1995bp}%
    \ifx\svgscale\undefined%
      \relax%
    \else%
      \setlength{\unitlength}{\unitlength * \real{\svgscale}}%
    \fi%
  \else%
    \setlength{\unitlength}{\svgwidth}%
  \fi%
  \global\let\svgwidth\undefined%
  \global\let\svgscale\undefined%
  \makeatother%
  \begin{picture}(1,0.61904765)%
    \lineheight{1}%
    \setlength\tabcolsep{0pt}%
    \put(0,0){\includegraphics[width=\unitlength,page=1]{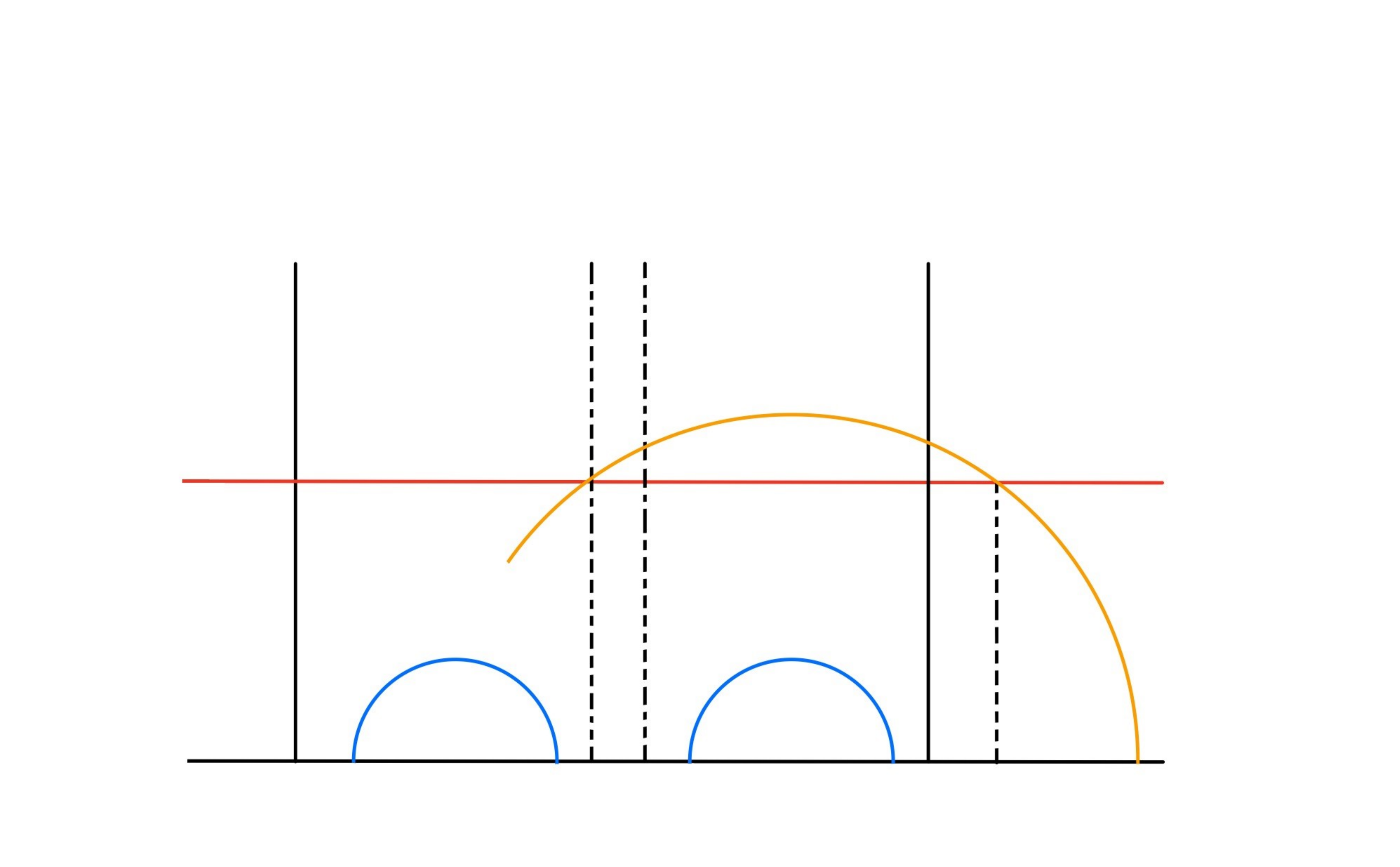}}%
    \put(0.81132188,0.04674907){\color[rgb]{0.05490196,0.05490196,0.05490196}\makebox(0,0)[lt]{\lineheight{1.25}\smash{\begin{tabular}[t]{l}$\xi$\end{tabular}}}}%
    \put(0.84751472,0.2722){\color[rgb]{0.05490196,0.05490196,0.05490196}\makebox(0,0)[lt]{\lineheight{1.25}\smash{\begin{tabular}[t]{l}$\partial \mathcal{HD}_\infty(k)$\end{tabular}}}}%
    \put(0.42300332,0.32045706){\color[rgb]{0.05490196,0.05490196,0.05490196}\makebox(0,0)[lt]{\lineheight{1.25}\smash{\begin{tabular}[t]{l}$\overset{(n)}{\ldots}$\end{tabular}}}}%
  \end{picture}%
\endgroup%

	\end{normalsize}
	\caption{A lift of an $(n,k)$-excursion of a geodesic ray on $\mathbb{H}$}
	\label{excur1}
\end{figure}



\begin{prop}[\textbf{Corollary 2.2 in \cite{Vo}}] \label{smalength}
Let $ \sigma' $ be an $ (n,k) $-excursion of a geodesic ray $ \sigma $ into the cusp $ \mathcal{C}_{p^{+}}(k) $ with horocycle boundary of length $ L $. We have 
$$ \arg\sinh(\frac{L}{2}\omega(\sigma))\leqslant\frac{1}{2}l(\sigma')<
\arg\sinh(\frac{L}{2}\omega(\sigma)+1). $$	
\end{prop}
Since $ L=\frac{1}{k} $, $ \omega(\sigma)=n $ and for every $ x\in\mathbb{R} $, since $ \arg\sinh(x)=\log{(x+\sqrt{x^2+1})}$, for large $x$ we have
$ \arg\sinh(x)\sim\log(2x) $ we have for every larger integer $ |n|\geqslant 2k $
\begin{equation}
l(\sigma')\sim d(0,p^{n}.0)\sim2\log\Big(\frac{n}{k}\Big).
\end{equation}

\subsection{Geometric encoding of the limit set}
In this subsection let us firstly introduce three alphabets of the limit set $ \Lambda $ of the group $ \Gamma $. The first one is $ \mathcal{A}_{1}=\{h^{\pm1}, p^{\pm1}\} $ constituted by the generators of $ \Gamma $ and their inverses. Every element $ \gamma\in\Gamma\setminus\{Id\}$ is written as follows $ \gamma=s_{1}s_{2}...s_{l} $ where $ s_{i}\in\mathcal{A}_{1} $ and $ s_{i+1}\neq s_{i}^{-1} $. So there is a bijection between the set of finite reduced sequences $ \{(s_{1},...s_{l}):l\in\mathbb{N}\} $ and the set $ \Gamma\setminus\{Id\} $. Consider the set of infinite sequences $ S=\{s=(s_{i})_{i\geq1}/s_{i}\in\mathcal{A}_{1},\; s_{i+1}\neq s_{i}^{-1}\} $ and define the function $ f:S\rightarrow\Lambda $ such that the infinite reduced sequence $s=(s_{i})_{i\geq1} $ is associated to the limit point $ \xi $ defined by $$ \xi=\lim_{m\rightarrow+\infty}\gamma_{m}(0)$$ where 
$ \gamma_{m}=s_{1}s_{2}...s_{m} $. The function $ f $ is  surjective and not injective. Each parabolic limit point is encoded (non-uniquely) by the sequences $ (s_{i})_{i\geq1} $ in $ S $ which are constant for large $ i $, and whose repeated term is a parabolic isometry. Each radial limit point is encoded by a unique sequence $ (s_{i})_{i\geq1} $ for which, if the term $ s_{i} $ is a parabolic isometry, there exists $ j>i $ such that $ s_{j}\neq s_{i} $. For more details we refer to \cite{dalbo}.\\
Let us consider a second alphabet $ \mathcal{A}_{2} $ whose elements are the powers of the isometries $ h $ and $ p $. So a reduced word representing a limit point $ \xi $ is written as follow
$$ \xi= h^{m_1}p^{n_1}\ldots h^{m_j}p^{n_j}\ldots $$
where the powers $ m_{j} $ and $ n_{j} $ are elements of $ \mathbb{Z}^{\ast} $. We write $ \mathcal{A}_{2}=\{h^{\mathbb{Z}^{\ast}},p^{\mathbb{Z}^{\ast}}\} $. Here also every element $ \gamma $ of the set $ \Gamma\setminus\{Id\} $ is uniquely written as follows $$\gamma=c_{1}...c_{l},\;\;\mbox{where}\;\;c_{i}\in\mathcal{A}_{2}\;\;\mbox{and}\,\,c_{i+1}\neq c_{i}^{-1}. $$
Let us introduce our third alphabet. For this we consider a compact region $ W $ of the surface $ \Sigma $ and a cusp $ \mathcal{C}_{p^{+}} $ such that $ N(\Sigma)=W\bigcup\mathcal{C}_{p^{+}} $. Let 
$$ \xi=h^{m_1}p^{n_1}h^{m_2}p^{n_2}h^{m_3}p^{n_3}h^{m_4}p^{n_4}...h^{m_{j-1}}p^{n_{j-1}}h^{m_j}p^{n_j}
...$$ be an element of $ \Lambda $ written in the alphabet $ \mathcal{A}_{2} $ and $ (\gamma_{j}) $ be the sequence of elements of $ \Gamma $ such that 
$$ \gamma_{j}=h^{m_1}p^{n_1}...h^{m_j}p^{n_j}\;\;\mbox{and}\;\;\gamma_{j}(0)\rightarrow\xi. $$
Denote $ N $ the minimum, if there exists, of the absolute value of the powers of $ p $ such that 
$$ \gamma_{N}(0)=h^{m_1}p^{n_1}...p^{N}(0)\notin W\;\;\mbox{and so}\;\;  h^{m_1}p^{n_1}...p^{N}(0)\in\mathcal{C}_{p^{+}}. $$
According to the compact set $ W $, we rewrite the isometry $ \gamma_{j} $ as follow
$$ \gamma_{j}=\gamma_{q}=\omega_{1}p^{r_{1}}\omega_{2}p^{r_{2}}...\omega_{q}p^{r_{q}} $$
where for every $ i\in[1,q]\;\;\mbox{and for every}\;\;k\in[1,s_{i}]$ 
$$|r_{i}|\geq N\;\;\mbox{and}\;\;\omega_{i}=h^{j_{1}}p^{l_{1}}h^{j_{2}}p^{l_{2}}...
p^{l_{s_{i}-1}}h^{j_{s_{i}}}\;\;\mbox{with}\;\;|l_{k}|<N\;\;\mbox{and}\;\;
j_{k}\in\mathbb{Z}^{\ast}. $$

Now since $ \Lambda_{p}\subset\Lambda_{\infty} $ let us examine the intersection $ \Lambda_{\infty}\cap\Lambda_{r} $. For this one knows that the radial limit set $ \Lambda_{r} $ is subdivided into two disjoints subsets. The \textit{uniformly} or \textit{bounded radial limit set} denoted $ \Lambda_{br} $, consisting of the limit points $ \xi\in\Lambda $ for which there exist a geodesic ray $ \tilde{\sigma} $ of $ \mathbb{H} $ ending at $ \xi $ and a constant $ c>0 $ such that $ d(\tilde{\sigma}(t),\Gamma(0))<c $, for all $ t>0 $. Note that this is equivalent to saying that the trajectory $\pi([\tilde{\sigma}(0),\xi))$ lays entirely in a compact subset $W\subset\Sigma$. The term $ \pi $ denotes the natural projection of $ \mathbb{D} $ onto $ \Sigma $. For instance, fixed points of hyperbolic isometries are bounded radial limit points. The complement of subset $\Lambda_{br} $ in $ \Lambda_{r} $ is called the \textit{unbounded radial limit set} of $ \Gamma $ and it is denoted $ \Lambda_{ur} $. Hence $ \Lambda_{r}=\Lambda_{br}\sqcup\Lambda_{ur} $.

\begin{lem}
	Every element of $ \Lambda_{br} $ is not a point of $\Lambda_{\infty}$ i.e 
	$$ \Lambda_{br}\cap\Lambda_{\infty}=\emptyset. $$
\end{lem}
\begin{proof}
From the definition of bounded radial point, if a ray $\sigma(t)$ has a lift directed to $\tilde\sigma(\infty)\in\Lambda_{br}$, then there exists a compact set $W_\sigma$ in the surface, such that $\sigma(t)\in W_\sigma$ for all $t\geqslant0$. So if we compute the integral of equation \ref{eq_divergencia}:
	
	\[\lim_{T\rightarrow+\infty}\frac{1}{T}\int_{0}^{T}\chi_{W_{\sigma}}(\sigma(t))dt=
	\lim_{T\rightarrow+\infty}\frac{1}{T}\int_{0}^{T}1dt=1,\]
	
	so $\tilde\sigma(\infty)$ is not a diverging on average limit point. 
	
\end{proof} 
In light of the above proposition, the bulk of the work that remains to be done is to analyze the intersection $ \Lambda_{br}^{\infty}=\Lambda_{\infty}\cap\Lambda_{ur} $.

\begin{prop}\label{radunb}
	A limit point $ \xi $ belongs to the set $ \Lambda_{ur} $ iff $\xi$ is radial and the sequence of powers of the parabolic isometry present in its coding is unbounded.
\end{prop}

\begin{proof}
	The fact that $\xi$ is radial translates into its coding as follows
	$$
	\xi= h^{m_1}p^{n_1}\ldots h^{m_j}p^{n_j}\ldots,
	$$
	where the sequence $m_j$ cannot be constant equal to $0$ from a certain point (see \cite{Dalbo-Starkov}).
	
	 We will now see what happens with the sequence $n_j$. Suppose $n_j$ bounded, then there exists $N$ such that $n_j\leq N$ for all $j\in \mathbb N$. If we now consider the projection in the surface of a ray starting from $i\in \mathbb{H}^2$ and directed towards $\xi$, the highest horocycle that the ray can possibly meet in the surface will be the projection of the horocycle $\mathcal{H}_{\infty}(N)$ defined by all the points with the same imaginary part of the geodesic joining $i$ and $p^N(i)$, see Figure \ref{dibujo8}. As a result, we just have to consider the compact set defined by intersection of complementary of the horodisc defined by $\mathcal{HD}_{\infty}(N)$ and the Nielsen region to obtain a compact set that entirely contains the ray $\pi([i,\xi)$, which is a contradiction with $\xi$ being an unbounded radial limit point.
\end{proof}

\begin{figure}[htbp]
	\centering
	\begin{normalsize} 
		\def\svgwidth{11cm}
\begingroup%
  \makeatletter%
  \providecommand\color[2][]{%
    \errmessage{(Inkscape) Color is used for the text in Inkscape, but the package 'color.sty' is not loaded}%
    \renewcommand\color[2][]{}%
  }%
  \providecommand\transparent[1]{%
    \errmessage{(Inkscape) Transparency is used (non-zero) for the text in Inkscape, but the package 'transparent.sty' is not loaded}%
    \renewcommand\transparent[1]{}%
  }%
  \providecommand\rotatebox[2]{#2}%
  \newcommand*\fsize{\dimexpr\f@size pt\relax}%
  \newcommand*\lineheight[1]{\fontsize{\fsize}{#1\fsize}\selectfont}%
  \ifx\svgwidth\undefined%
    \setlength{\unitlength}{1182bp}%
    \ifx\svgscale\undefined%
      \relax%
    \else%
      \setlength{\unitlength}{\unitlength * \real{\svgscale}}%
    \fi%
  \else%
    \setlength{\unitlength}{\svgwidth}%
  \fi%
  \global\let\svgwidth\undefined%
  \global\let\svgscale\undefined%
  \makeatother%
  \begin{picture}(1,0.6734348)%
    \lineheight{1}%
    \setlength\tabcolsep{0pt}%
    \put(0,0){\includegraphics[width=\unitlength,page=1]{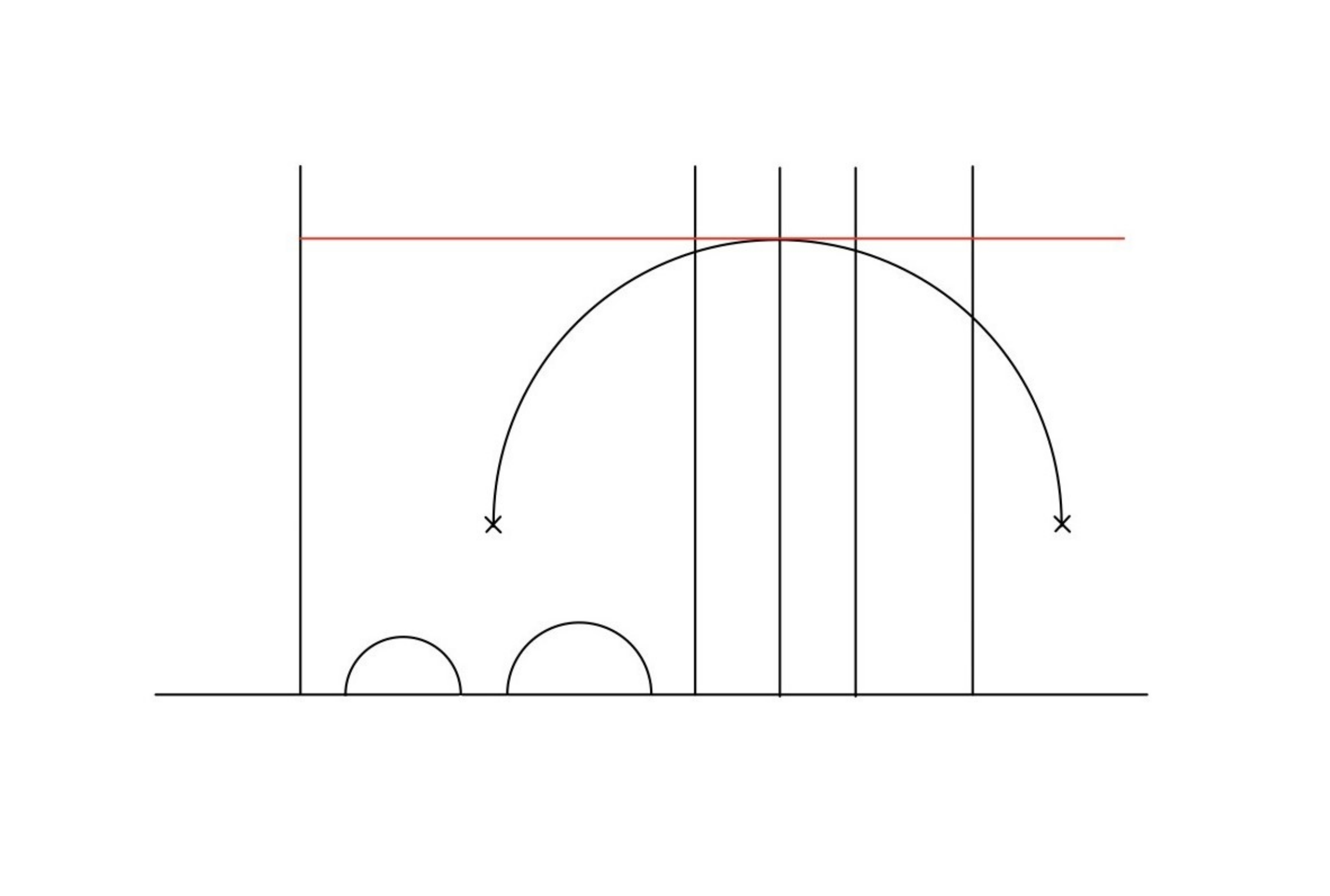}}%
    \put(0.38470269,0.2780687){\color[rgb]{0.05490196,0.05490196,0.05490196}\makebox(0,0)[lt]{\lineheight{1.25}\smash{\begin{tabular}[t]{l}$i$\end{tabular}}}}%
    \put(0.81944135,0.2780687){\color[rgb]{0.05490196,0.05490196,0.05490196}\makebox(0,0)[lt]{\lineheight{1.25}\smash{\begin{tabular}[t]{l}$p^N(i)$\end{tabular}}}}%
    \put(0.16651312,0.49379752){\color[rgb]{0.05490196,0.05490196,0.05490196}\makebox(0,0)[lt]{\lineheight{1.25}\smash{\begin{tabular}[t]{l}$\tilde{W}_N$\end{tabular}}}}%
    \put(0.67179429,0.35025172){\color[rgb]{0.05490196,0.05490196,0.05490196}\makebox(0,0)[lt]{\lineheight{1.25}\smash{\begin{tabular}[t]{l}$\ldots$\end{tabular}}}}%
  \end{picture}%
\endgroup%

	\end{normalsize}
	\caption{The highest horocycle that $\pi([i,\xi))$ can meet is $\pi(\mathcal{HD}_{\infty}(N))$).}
	\label{dibujo8}
 \end{figure}
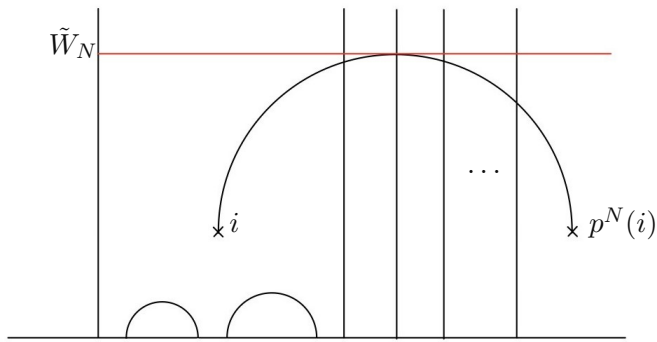

\subsection{Proof of theorem \ref{average}}\label{fundcros}
Let $ \sigma $ to be a geodesic ray of the hyperbolic surface $ \Sigma $ parameterized by hyperbolic arc-length such that one of its lifts, $ \tilde{\sigma} $ on $ \mathbb{D} $ verifies $ \tilde{\sigma}(0)=0 $ and $ \tilde{\sigma}(\infty)=\xi\in\Lambda_{ur} $. Consider now an increasing sequence of compacts $ (W_{k})_{k\geq1} $ of the surface $ \Sigma $ such that $ \pi(0)\in W_{1} $ and $ N(\Sigma)=W_{k}\cup\mathcal{C}_{p^{+}}(k) $. Every $ W_{k} $ has piecewise $ C^{1} $ boundaries. Let $ T_{k} $ be a time such that $ \sigma(T_{k})=w_{k}\in W_{k} $. We put 
$$ l_{W_{k}}=l(\sigma_{|[0,T_{k}]}\cap W_{k})\;\;\mbox{and}\;\; l_{C_{p^{+}}(k)}=l(\sigma_{|[0,T_{k}]}\cap\mathcal{C}_{p^{+}}(k)) $$
where $ l(\alpha) $ designs the length of a curve $ \alpha $ in the surface. Since the geodesic $ \sigma $ is parameterized by hyperbolic arc length we have :
\begin{equation}
\frac{1}{T_{k}}\int_{0}^{T_{k}}\chi_{W_k}(\sigma(t))dt=\frac{l_{W_{k}}}{l_{W_{k}}+l_{C_{p^{+}}(k)}}=
\frac{1}{1+\frac{l_{C_{p^{+}}(k)}}{l_{W_{k}}}}.
\end{equation} 
So by definition $ \xi\in\Lambda_{\infty} $ if only if for every integer $ k\geq0 $, the ratio 
$ \frac{l_{W_{k}}}{l_{C_{p^{+}}(k)}} $ goes to $ 0 $ whenever $ T_{k} $ goes to infinity.\\

\begin{figure}[htbp]
	\centering
	\begin{normalsize} 
		\def\svgwidth{10cm}
\begingroup%
  \makeatletter%
  \providecommand\color[2][]{%
    \errmessage{(Inkscape) Color is used for the text in Inkscape, but the package 'color.sty' is not loaded}%
    \renewcommand\color[2][]{}%
  }%
  \providecommand\transparent[1]{%
    \errmessage{(Inkscape) Transparency is used (non-zero) for the text in Inkscape, but the package 'transparent.sty' is not loaded}%
    \renewcommand\transparent[1]{}%
  }%
  \providecommand\rotatebox[2]{#2}%
  \newcommand*\fsize{\dimexpr\f@size pt\relax}%
  \newcommand*\lineheight[1]{\fontsize{\fsize}{#1\fsize}\selectfont}%
  \ifx\svgwidth\undefined%
    \setlength{\unitlength}{1995bp}%
    \ifx\svgscale\undefined%
      \relax%
    \else%
      \setlength{\unitlength}{\unitlength * \real{\svgscale}}%
    \fi%
  \else%
    \setlength{\unitlength}{\svgwidth}%
  \fi%
  \global\let\svgwidth\undefined%
  \global\let\svgscale\undefined%
  \makeatother%
  \begin{picture}(1,0.75288224)%
    \lineheight{1}%
    \setlength\tabcolsep{0pt}%
    \put(0,0){\includegraphics[width=\unitlength,page=1]{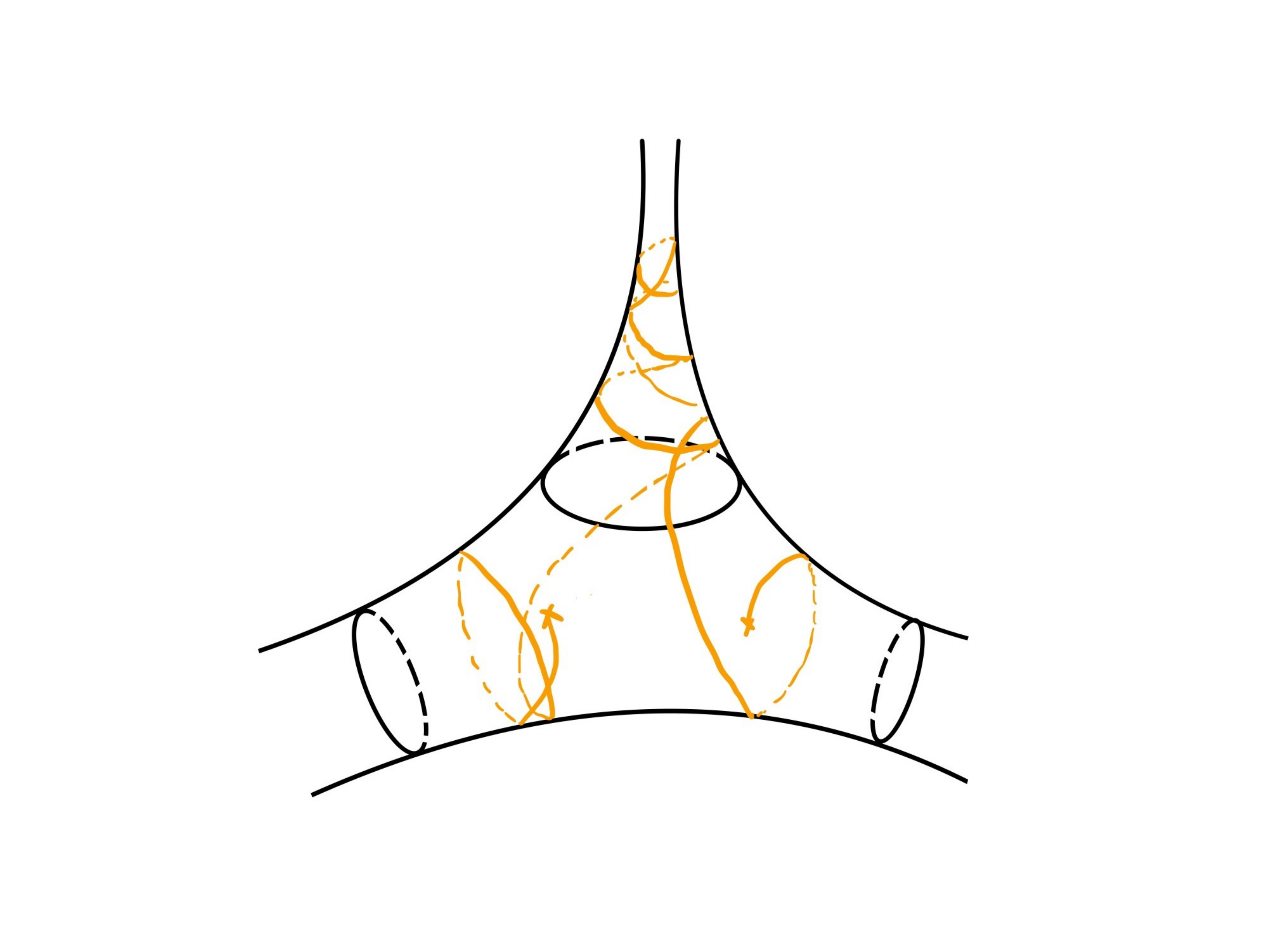}}%
    \put(0.45393021,0.26960706){\color[rgb]{0.05490196,0.05490196,0.05490196}\makebox(0,0)[lt]{\lineheight{1.25}\smash{\begin{tabular}[t]{l}$\pi(0)$\end{tabular}}}}%
    \put(0.59790402,0.37965072){\color[rgb]{0.05490196,0.05490196,0.05490196}\makebox(0,0)[lt]{\lineheight{1.25}\smash{\begin{tabular}[t]{l}$C_{p^+}(k)$\end{tabular}}}}%
    \put(0.59790402,0.24576421){\color[rgb]{0.05490196,0.05490196,0.05490196}\makebox(0,0)[lt]{\lineheight{1.25}\smash{\begin{tabular}[t]{l}$w=\sigma(T)$\end{tabular}}}}%
  \end{picture}%
\endgroup%

	\end{normalsize}
	\caption{The geodesic segment $\sigma|_{[0,T]}$ winds in a cusp and returns in the corresponding compact}
	\label{excur3}
\end{figure}
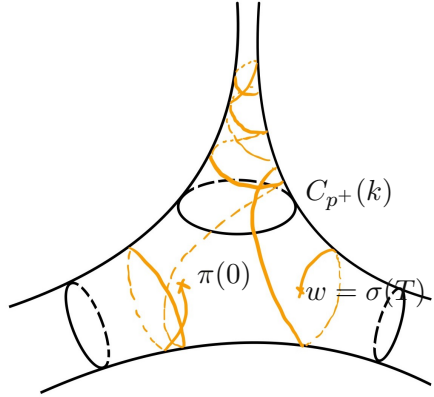

To prove the theorem \ref{average} we start by establishing a lemma and a corollary.


\begin{lem}\label{tid1}
	For any $ \theta>0 $, there is some positive constant $ C(\theta) $ such that for any hyperbolic triangle with sides $ a,\;b,\;c $ and opposite angles $ \theta_a,\;\theta_b,\;\theta_c $ with $ \theta_c\geq\theta $ we have $$ a+b-C(\theta)\leqslant c\leqslant a+b. $$
\end{lem}
\begin{proof}
	For the first inequality, the law of the cosines (see \cite{Buser}) gives
	
	$$\cosh c=\cosh a\cosh b-\sinh a\sinh b\cos \theta_c$$
		
		We first assume $\theta_c\geq \frac{\pi}{2}$.
		We obtain $$\cosh c\geq \cosh a\cosh b-\sinh a\sinh b \geq\frac{\cosh(a+b)}{2}$$
		Then $$ c+\cosh^{-1}(2)\geq a+b.$$
We set in this case $ C(\theta)=\cosh^{-1}(2) $.\\		
		Suppose now $\theta_c<\frac{\pi}{2}$. Then 
		
		$$\cosh c\geq \cosh a\cosh b-\cosh a\cosh b \cos(\theta_c)$$
		and 
		$$ \cosh c\geq \frac{\cosh(a+b)}{2}(1-\cos \theta).$$
		So 
		$$c+\cosh^{-1}\left(\frac{2}{1-\cos \theta}\right)\geq a+b.$$
		
		It suffices here to set $C(\theta)=\cosh^{-1}(\frac{2}{1-\cos\theta})$.
	
	The triangular inequality is the second one.

\end{proof}

\begin{figure}[htbp]
	\centering
	\begin{normalsize} 
		\def\svgwidth{12cm}
		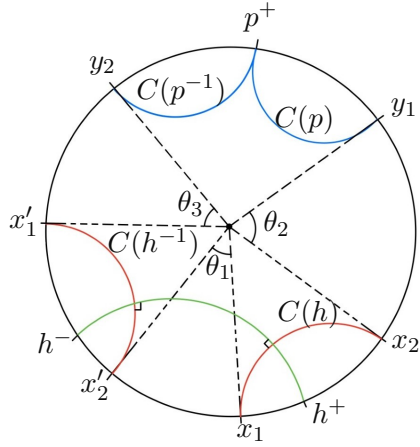
	\end{normalsize}
	\caption{Notation for Corollary \ref{tid}.}
	\label{dibujo9}
\end{figure}

Let us redesign the bisectors of the segments $ [0,g(0)] $ where $ g\in\mathcal{A}_{1} $ in putting
$$ C(h)=(x_{1},x_{2}),\;C(h^{-1})=(x_{1}',x_{2}'),\;C(p)=(p^{+},y_{1})\;\mbox{and}\;C(p^{-1})=(p^{+},
y_{2}). $$
See Figure \ref{dibujo9}. We denote
$$ \theta_{1}:=\;\mbox{the angle between the geodesic rays}\;[0,x_{1})\;\mbox{and}\;[0,x_{2}'), $$
$$ \theta_{2}:=\;\mbox{the angle between the geodesic rays}\;[0,x_{2})\;\mbox{and}\;[0,y_{1}), $$
$$ \theta_{3}:=\;\mbox{the angle between the geodesic rays}\;[0,x_{1}')\;\mbox{and}\;[0,y_{2}). $$
Consider the sub-alphabet $ \mathcal{A}_{2}^{+}=\{h^{\mathbb{Z}^{\ast}},p^{\mathbb{N}^{\ast}}\} $ of $ \mathcal{A}_{2} $ whose powers of the parabolic element $ p $ are nonegative. For every two elements $ \gamma^{1} $ and $ \gamma^{2} $ of the group $ \Gamma $ such that $\gamma^{1}=c_{1}...c_{l},\;\; $, $\gamma^{2}=c_{1}'...c_{s}',\;\;\mbox{where}\;\;c_{i}, c_{i}'\in\mathcal{A}_{2}^{+} $ and $ c_{1} $, $c_{1}' $ are distinct powers of the hyperbolic isometry $ h $, we have.

\begin{cor}\label{tid}
There exists a constant $ C=C(\theta_{0})>0 $ such that
$$ d(\gamma^{1}(0),\gamma^{2}(0))\geqslant\ d(\gamma^{1}(0),0)+d(0,\gamma^{2}(0))-C. $$
\end{cor}
\begin{proof}
We have on one hand
$$ c_{1}...c_{l}(0)\in D(c_{1})\;\;\mbox{and}\;\;c_{1}...c_{s}(0)\in D(c_{1}'). $$
On other hand the discs $ D(c_{1}) $ and $ D(c_{1}') $ are disjoint since $ c_{1}\neq c_{1}' $.	
So we have the angle
$$ \widehat{(\gamma^{1}(0),0,\gamma^{2}(0))}\geqslant\theta_{0}=\min_{1\leqslant i\leqslant3}\theta_{i}. $$
Hence from the lemma \ref{tid1}, we have the desired inequality.	
\end{proof}
\begin{proof}[Proof of the theorem \ref{average}]
Let us consider the writing of $ \xi\in\Lambda_{ur} $ in the alphabet $ \mathcal{A}_{3} $ meaning
$\xi=\omega_{1}p^{r_{1}}\omega_{2}p^{r_{2}}\omega_{3}p^{r_{3}}...,$
and $\gamma_q=\omega_{1}p^{r_{1}}...\omega_{q}p^{r_{q}}$ with $\gamma_q(0)\to\xi $ when $q$ goes to infinity. Recall that $l$ was the length of the geodesic segment of $\sigma$ from $\pi(0)$ to $\pi(\tilde\sigma(\R^+)\cap D(\gamma_q))$, where $\sigma$ is the geodesic starting at $\pi(0)$ and directed to $\xi$. In addition, we call $l_W$ the length of the arc of this geodesic segment which belongs in the compact $W$. 
Let $ \beta $ be the closed geodesic segment based on the point $ \pi(0) $ and which ends in $\pi(\gamma_q(0))$ and put
$ M_{1}=diam(W) $. We have
\begin{equation}\label{ineq}
l(\beta)-M_{1}\leqslant l\leqslant l(\beta)+M_{1}.
\end{equation}
Let $ \tilde{\beta} $ be the lift of $ \beta $ such that $ \tilde{\beta}=[0,\gamma_{q}(0)] $. We have
$$ l(\beta)=l(\pi([0,\gamma_{q}(0)]))=d(0,\gamma_{q}(0)). $$
And from the corollary \ref{tid} there exists a constant $ C>0 $ such that
\begin{equation}\label{lengthloop}
\sum_{i=1}^{q}d(0,\omega_{i}(0))+
\sum_{i=1}^{q}\ln|r_{i}|-2qC\leqslant l(\beta)\leqslant\sum_{i=1}^{q}d(0,\omega_{i}(0))+
\sum_{i=1}^{q}d(0,p^{r_i}(0)),
\end{equation}
and	
\begin{equation}\label{lengthcomp}
\sum_{i=1}^{q}d(0,\omega_{i}(0))-qC\leqslant l_{W}\leqslant\sum_{i=1}^{q}d(0,\omega_{i}(0)). 
\end{equation}
So the triple inequalities \ref{ineq}, \ref{lengthloop} and \ref{lengthcomp} give
\begin{equation*}
\frac{1-\frac{qC}{\sum_{i=1}^{q}d(0,\omega_{i}(0))}}{1+\frac{\sum_{i=1}^{q}d(0,p^{r_i}(0))}{
		\sum_{i=1}^{q}d(0,\omega_{i}(0))}+\frac{M_{1}}{\sum_{i=1}^{q}d(0,\omega_{i}(0))}}\leqslant
\frac{l_{W}}{l},
\end{equation*}

\begin{equation}
\frac{l_{W}}{l}\leqslant
\frac{1}{1+\frac{\sum_{i=1}^{q}d(0,p^{r_i}(0))}{\sum_{i=1}^{q}d(0,\omega_{i}(0))}
-\frac{2qC}{\sum_{i=1}^{q}d(0,\omega_{i}(0))}-\frac{M_{1}}{\sum_{i=1}^{q}d(0,\omega_{i}(0))}}.
\end{equation}
Since $ d(0,\omega_{i}(0))\geqslant d(0,h(0))=l(h) $ (we are supposing that there is at least a hyperbolic isometry $h$ in the coding of $\omega_i$, otherwise the limit point would not be radial), we have $ \frac{2qC}{\sum_{i=1}^{q}d(0,\omega_{i}(0))}\leqslant\frac{2C}{l(h)} $. Then, as $d(0,p^{r_i}(0))\sim 2\ln|r_i|$ when $q$ goes to infinity, we obtain 
$$ \frac{l_{W}}{l}\rightarrow 0 \Leftrightarrow  
\lim_{q\rightarrow+\infty}\frac{\sum_{i=1}^{q}d(0,\omega_{i}(0))}{\sum_{i=1}^{q}2\ln|r_{i}|}=0.
$$
Hence
$$ \xi\in\Lambda_{\infty} \Leftrightarrow  
\lim_{q\rightarrow+\infty}\frac{\sum_{i=1}^{q}d(0,\omega_{i}(0))}{\sum_{i=1}^{q}2\ln|r_{i}|}=0.
$$
\end{proof}
Before going on, let us exemplify that result.
\subsection{Examples}

	Let $\xi=\omega_{1}p^{r_{1}}\omega_{2}p^{r_{2}}...\omega_{q}p^{r_{q}}...\in\Lambda $ and $ \gamma_{q}=\omega_{1}p^{r_{1}}\omega_{2}p^{r_{2}}...\omega_{q}p^{r_{q}}\in\Gamma $.
 Consider the two following scenarios:
\begin{itemize}
	\item Case 1: $\xi=hphp^{2}hp^{3}hp^{4}...$ , then for every compact of the sequence $W_k$ we have that $\frac{\sum_{i=1}^{q}d(0,\omega_{i}(0))}{\sum_{i=1}^{q}2\ln|r_{i}|} \approx \frac{ql(h)}{\sum_1^q{2\ln i}}$, which converges to $0$ when $q$ goes to infinity. 
	\item Case 2: $\xi=hph^2p^{2}h^3p^{3}h^4p^{4}...$, then for every compact of the sequence $W_k$ we have that $\frac{\sum_{i=1}^{q}d(0,\omega_{i}(0))}{\sum_{i=1}^{q}2\ln|r_{i}|} \approx \frac{\sum_1^q i l(h)}{\sum_1^q{2\ln i}}=\frac{q^2(q+1)l(h)}{2ln(q!)}$, which diverges when $q$ goes to infinity. 
\end{itemize}

We deduce from these examples that the set $ \Lambda_{\infty} $ is uncountable but it is a strict subset of $ \Lambda_{ur}\cup\Lambda_{p} $. As the Schottky group $\Gamma$ is a second kind group, this group is divergent, and then, the Lebesgue measure of its radial set is zero. Now, since $ \Lambda_{ur} \cup \Lambda_p $ is a subset of $\Lambda_r\cup \Lambda_p$ and this set also has zero Lebesgue measure, we deduce that so does $\Lambda_\infty$. Thus we aught to study its fractal.

\section{Hausdorff dimension of the diverging on average set}
In order to determine the size of the set $ \Lambda_{\infty} $ let us first define a distance on $ \mathbb{S}^{1} $.
\begin{defn}
Let $ \xi\in\mathbb{S}^{1} $ and $ z,\;z' $ two points of $ \mathbb{D} $. We call \textbf{Busemann cocycle} based on $ \xi $ between $ z $ and $ z' $ the quantity
$$ \beta_{\xi}(z,z'):=\lim_{t\rightarrow+\infty}(d(z,\xi_{t})-d(z',\xi_{t})) $$	where $ t\mapsto\xi_{t} $, is the parametrization of a geodesic ray ending at $ \xi $.
\end{defn}
\begin{defn}
	The \textbf{visual distance} $ d_{0} $ on $ \mathbb{S}^{1} $ is defined as
	$$ d_{0}(\xi,\eta)=\left\{
	\begin{array}{ll} e^{-<\xi,\eta>_{0}},\;\mbox{if}\;\;\xi\neq\eta,\\
	0,\;\mbox{if}\;\;\xi=\eta.
	\end{array}\right.$$
	where $ <\xi,\eta>_{0}=-\frac{1}{2}(\beta_{\xi}(0,z)+\beta_{\eta}(0,z)) $, for any $ z\in(\xi,\eta) $, called the Gromov product based on the point $ 0 $.
\end{defn}
For the following definitions and the proposition the reader is referred to Chapter 2 of Mattila's book \cite{matila} for details. 
\begin{defn}
Let $ F $ be a subset of $ \mathbb{S}^{1} $. For every $ \delta>0 $
\begin{enumerate}
	\item we call a $ \delta $-\textbf{cover} of $ F $ every collection of balls $ \{B_{i}(x_{i},r_{i})\}_{i\geqslant1} $ of $ \mathbb{S}^{1} $ such that $ r_{i}\leqslant\delta $ and $ F\subseteq\bigcup_{i=1}^{+\infty}B_{i} $.
	\item we define $ \mathcal{H}_{\delta}^{s}(F):=\inf\big\{\sum_{i=1}^{+\infty}r_{i}^{s}: \{B_{i}(x_{i},r_{i})\}_{i\geqslant1}\;\mbox{is a}\;\delta\mbox{-cover of}\;F \big\} $.
	\item The $s$-\textbf{dimensional Hausdorff measure} of $ F $ is given by 
		$$ \mathcal{H}^{s}(F):=\lim_{\delta\rightarrow0}\mathcal{H}_{\delta}^{s}(F). $$
\end{enumerate}
\end{defn}
$ \mathcal{H}^{s}(F) $ can take any value between $ 0 $ and $ \infty $. If $ \mathcal{H}^{s}(F) $ is finite then $ \mathcal{H}^{t}(F)=0 $ for every $ t>s $. So there is a critical value of $ s $ where $ \mathcal{H}^{s}(F) $ jumps from $ \infty $ to 0.
\begin{defn}
	We call \textbf{Hausdorff dimension} of a subset $ F $ of $ \mathbb{S}^{1} $ the critical value
	$$ dim_{H}(F):=\sup\{s:\mathcal{H}^{s}(F)=\infty\}=\inf\{s:\mathcal{H}^{s}(F)=0\}. $$
\end{defn}

\begin{prop}
	Let $ F\subset\mathbb{S}^{1} $. The following hold:
	\begin{enumerate}
		\item $ 0\leqslant dim_{H}(F)\leqslant1 $.
		\item If $ E\subseteq F $, then $ dim_{H}(E)\leqslant dim_{H}(F) $.
		\item If $ \{F_{n}\}_{n\geqslant1} $ is a countable sequence of sets of $ \mathbb{S}^{1} $ then
		$$ dim_{H}\big(\bigcup_{n\geqslant1}F_{n}\big)=\sup\{dim_{H}(F_{n}), n\geqslant1\}. $$
And so if $ F $ is countable $ dim_{H}(F)=0 $.	
\end{enumerate}	

\end{prop}

Since $ \Lambda_{ur}^{\infty}=\Lambda_{\infty}\cap\Lambda_{ur} $ and the subset $ \Lambda_{p} $ is countable we have ought to determinate $ dim_{H}(\Lambda_{ur}^{\infty})$. We establish the theorem \ref{hausd} in two moments.

\subsection{The upper bound}
Let us cover the set $ \Lambda_{ur}^{\infty} $ by balls of $ \mathbb{S}^{1} $ indexed by a proper subset of the group $ \Gamma $. For this, we state the following two lemmas. For their proofs we refer to \cite{Bourdon} and \cite{BridsonHaefliger}.

\begin{lem}[Lemma 1.6.2 in \cite{Bourdon}]\label{kaim}
	Given $ r>0 $ and $ \gamma\in\Gamma $ there exists $ c(r)>1 $ such that we have
	$$ B(\xi_{0,\gamma (0)},c(r)^{-1}e^{-d(0,\gamma(0))})\subseteq
	\mathcal{O}(\gamma(0),r)\subseteq B(\xi_{0,\gamma (0)},c(r)e^{-d(0,\gamma(0))}), $$
	where $ \xi_{0,\gamma (0)} $ is the endpoint of the geodesic ray with origin $ 0 $ passing through $ \gamma_{q}(0) $ and 
	$$ \mathcal{O}(\gamma(0),r)=\{\eta\in\mathbb{S}^{1}:[0,\eta)\cap
	B(\gamma(0),r)\neq\emptyset\} $$
	it is the shadow at infinity from $ 0\in\mathbb{D} $ of the ball $ B(\gamma(0),r) $. And 
	$$ B(\xi_{0,\gamma (0)},c(r)e^{-d(0,\gamma(0))})=\{\eta\in\mathbb{S}^{1}:d_{0}(\eta,\xi_{0,\gamma (0)})<
	c(r)e^{-d(0,\gamma(0))}\}. $$
\end{lem}
Let $ \xi=\omega_{1}p^{r_{1}}\omega_{2}p^{r_{2}}...\omega_{q}p^{r_{q}}...\in\Lambda_{ur}^{\infty} $ and consider the piecewise geodesic curve $ \varUpsilon=([\gamma_{q+1}(0),\gamma_{q}(0)])_{q>0} $. This curve is $ C $-quasi-geodesic, where $ C $ is the constant of the corollary \ref{tid}. So we can state in our context the Morse Lemma established in a more general context, see \cite{BridsonHaefliger}, Chapter III.H, Theorem 1.13.
\begin{lem}\label{neighbor}.
There exists $ D>0 $ such that for any $ \xi=\omega_{1}p^{r_{1}}...\omega_{q}p^{r_{q}}...\in\Lambda_{ur}^{\infty} $ the piecewise geodesic curve $ \varUpsilon=([\gamma_{q+1}(0),\gamma_{q}(0)])_{q>0} $ is contained in a $ D $-neighborhood of the geodesic ray $ [0,\xi) $. 
\end{lem}


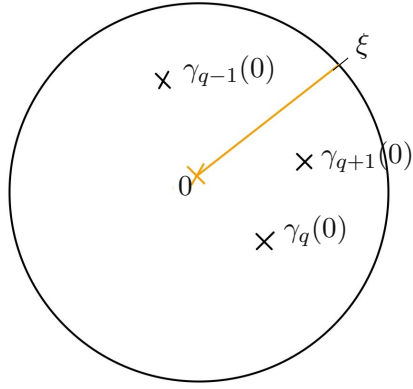
\begin{figure}[htbp]
	\centering
	\begin{normalsize} 
		\def\svgwidth{\textwidth}
\begingroup%
  \makeatletter%
  \providecommand\color[2][]{%
    \errmessage{(Inkscape) Color is used for the text in Inkscape, but the package 'color.sty' is not loaded}%
    \renewcommand\color[2][]{}%
  }%
  \providecommand\transparent[1]{%
    \errmessage{(Inkscape) Transparency is used (non-zero) for the text in Inkscape, but the package 'transparent.sty' is not loaded}%
    \renewcommand\transparent[1]{}%
  }%
  \providecommand\rotatebox[2]{#2}%
  \newcommand*\fsize{\dimexpr\f@size pt\relax}%
  \newcommand*\lineheight[1]{\fontsize{\fsize}{#1\fsize}\selectfont}%
  \ifx\svgwidth\undefined%
    \setlength{\unitlength}{1995bp}%
    \ifx\svgscale\undefined%
      \relax%
    \else%
      \setlength{\unitlength}{\unitlength * \real{\svgscale}}%
    \fi%
  \else%
    \setlength{\unitlength}{\svgwidth}%
  \fi%
  \global\let\svgwidth\undefined%
  \global\let\svgscale\undefined%
  \makeatother%
  \begin{picture}(1,0.57393487)%
    \lineheight{1}%
    \setlength\tabcolsep{0pt}%
    \put(0,0){\includegraphics[width=\unitlength,page=1]{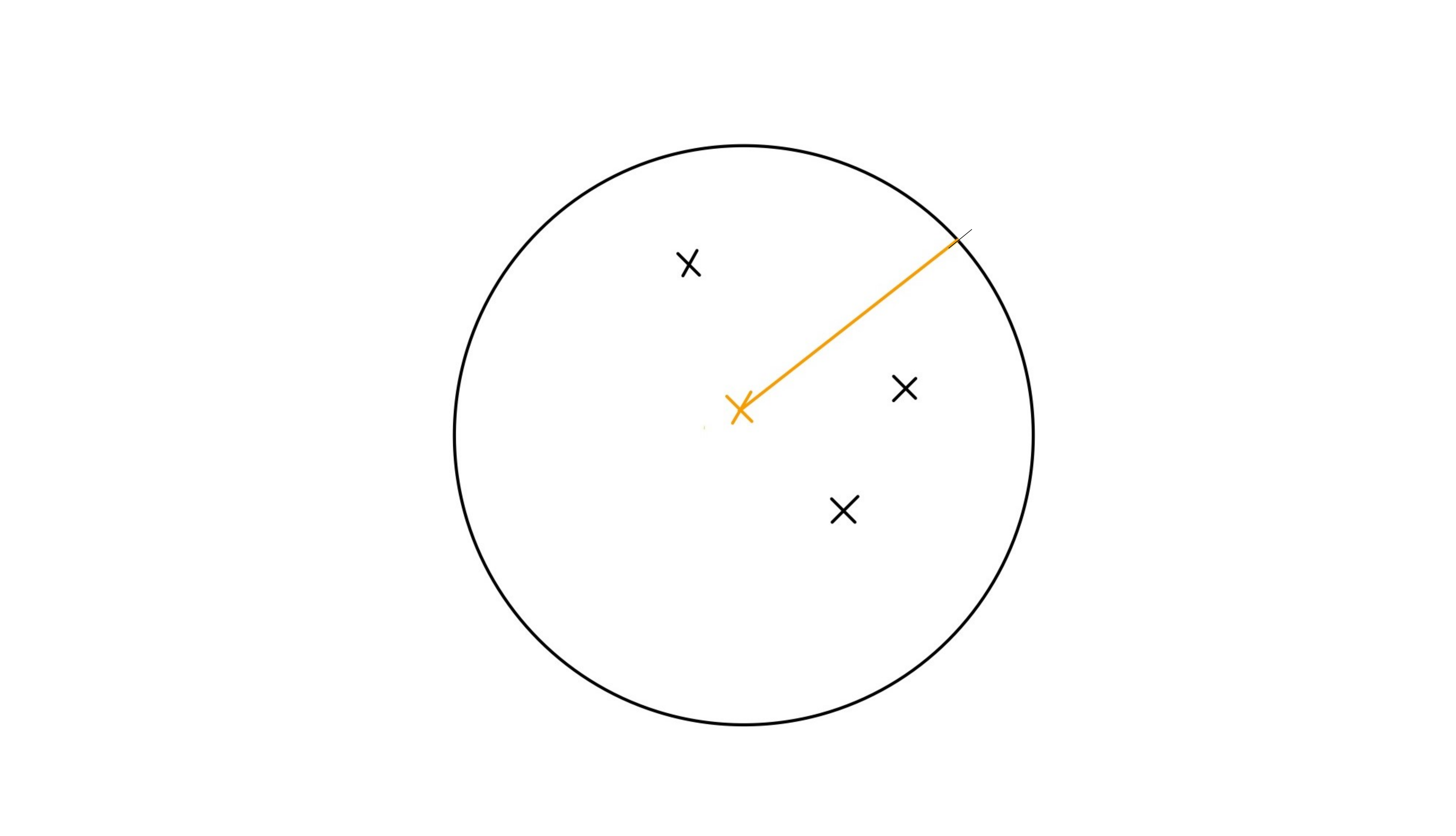}}%
    \put(0.49419446,0.39614096){\color[rgb]{0.05490196,0.05490196,0.05490196}\makebox(0,0)[lt]{\lineheight{1.25}\smash{\begin{tabular}[t]{l}$\gamma_{q-1}(0)$\end{tabular}}}}%
    \put(0.63933239,0.30542979){\color[rgb]{0.05490196,0.05490196,0.05490196}\makebox(0,0)[lt]{\lineheight{1.25}\smash{\begin{tabular}[t]{l}$\gamma_{q+1}(0)$\end{tabular}}}}%
    \put(0.60014515,0.22778102){\color[rgb]{0.05490196,0.05490196,0.05490196}\makebox(0,0)[lt]{\lineheight{1.25}\smash{\begin{tabular}[t]{l}$\gamma_q(0)$\end{tabular}}}}%
    \put(0.67561682,0.4215401){\color[rgb]{0.05490196,0.05490196,0.05490196}\makebox(0,0)[lt]{\lineheight{1.25}\smash{\begin{tabular}[t]{l}$\xi$\end{tabular}}}}%
    \put(0.48911468,0.27204808){\color[rgb]{0.05490196,0.05490196,0.05490196}\makebox(0,0)[lt]{\lineheight{1.25}\smash{\begin{tabular}[t]{l}$0$\end{tabular}}}}%
  \end{picture}%
\endgroup%

	\end{normalsize}
	\caption{$D$-neighborhood of a geodesic ray in the Poincare disc}
	\label{neigh}
\end{figure}

Let us $ \xi=\omega_{1}p^{r_{1}}...\omega_{q}p^{r_{q}}...\in\Lambda_{ur}^{\infty} $. We have $ \lim_{q\rightarrow+\infty}\frac{\sum_{i=1}^{q}
d(0,\omega_{i}(0))}{\sum_{i=1}^{q}2\ln|r_{i}|}=0 $. Then for every $ \varepsilon>0 $ there exists an integer $ q_{1} $ such that for any $ q\geqslant q_{1} $, 
\begin{equation}\label{ineq1}
\frac{\sum_{i=1}^{q}d(0,\omega_{i}(0))}{\sum_{i=1}^{q}2\ln|r_{i}|}\leqslant\varepsilon.
\end{equation}

In accordance with the notations established in Lemma \ref{kaim} denote by $ \xi_{0,\gamma_q (0)} $ the endpoint of the geodesic ray with origin $ 0 $ passing through $ \gamma_{q}(0) $. Given the constant $ r=D $, which is independent of $ \xi $, of the lemma \ref{neighbor}, we deduce from the lemma \ref{kaim} that there exists $ C_{2}=c(D)>0 $ such that $ \xi\in B(\xi_{0,\gamma_q (0)},C_{2}e^{-d(0,\gamma_{q}(0))}) $. And since $ \gamma_{q}(0) $ tends to $ \xi $, for any $ \delta>0 $ there exists an integer $ q_{2} $ such that for every $ q\geq q_{2} $ one has

\begin{equation}\label{ineq2}
2C_{2}e^{-d(0,\gamma_{q}(0))}\leqslant\delta. 
\end{equation}

In combining the inequalities \ref{ineq1} and \ref{ineq2}, for any $ \delta>0 $ we have
\begin{equation}\label{ineq3}
\frac{\sum_{i=1}^{q}d(0,\omega_{i}(0))}{\sum_{i=1}^{q}2\ln|r_{i}|}\leqslant\delta\;\;\mbox{and}\;\;
2C_{2}e^{-d(0,\gamma_{q}(0))}\leqslant\delta
\end{equation}
Now let us consider, for every positive integer $ N $ and real $ \delta>0 $, the subset of the group $ \Gamma $ denoted $ \Gamma_{\delta} $ defined as follows: an element $ \gamma\in\Gamma $ belongs on $ \Gamma_{\delta} $ iff there exists $ q\in\mathbb{N}^{\ast} $ such that
$$ \gamma=\gamma_{q}=\omega_{1}p^{r_{1}}...\omega_{q}p^{r_{q}}\;\;
\mbox{and verifies the inequations}\;\ref{ineq3}.$$
So for any $ \delta>0 $ the family 
$$ \Big\{B(\xi_{0,\gamma_q (0)},C_{2}e^{-d(0,\gamma_{q}(0))}):\gamma_{q}\in\Gamma_{\delta}\Big\}$$ 
constitutes a covering of the set $ \Lambda_{ur}^{\infty} $ with diameters less than $ \delta $. So for every $ s>0 $, we have $$ \mathcal{H}_{\delta}^{s}(\Lambda_{ur}^{\infty})\leqslant C_{2}^{s}
\sum_{\gamma\in\Gamma_{\delta}}e^{-sd(0,\gamma(0))}=C_{2}^{s}P_{\Gamma_{\alpha}^{\delta}}(s). $$
where $ P_{\Gamma_{\delta}}(s)=\sum_{\gamma\in\Gamma_{\delta}}e^{-sd(0,\gamma.0)} $. 
For every integer $ t $ consider the following subsets of $ \Gamma_{\delta} $ 
$$ A_{t}^{\delta}=\{\gamma\in\Gamma_{\delta}:\delta t\leqslant d(0,\gamma.0)\leqslant\delta(t+1)\}, $$
$ \Omega_{t}:= $ the set of terms $\omega$ on $ \gamma A_{t}^{\delta} $ such that 
$ \omega=h^{m_{1}}p^{n_{1}}...h^{m_{s}}p^{n_{s}} $ with $ |n_{i}|<N $\\
and\\
$ P_{t}:= $ the set of terms $ p^{r} $ on $ \gamma\in A_{t}^{\delta} $ of the form such that $ |r|\geq N $.\\
We have successively
\begin{align*}
P_{\Gamma_{\delta}}(s)
&=&\sum_{q=1}^{+\infty}\;\;\;\sum_{t=1}^{+\infty}\;\;\;\sum_{\gamma\in A_{t}^{\delta}}\;\;e^{-sd(0,\gamma.0)}\\
&\leqslant&\sum_{q=1}^{+\infty}\sum_{t=1}^{+\infty}\sum_{\gamma\in A_{t}^{\delta}}
exp\Big(2qsC-s\sum_{i=1}^{q}d(0,\omega_{i}.0)\Big)\times exp\Big(-s\sum_{i=1}^{q}d(0,p^{r_{i}}.0)\Big)\\
&=&\sum_{q=1}^{+\infty}\sum_{t=1}^{+\infty}e^{2qsC}\times
\Bigg(\sum_{\omega\in\Omega_{t}}e^{-sd(0,\omega.0)}\Bigg)^{q}
\Bigg(\sum_{p^{r}\in P_{t}}e^{-sd(0,p^{r}.0)}\Bigg)^{q}\\
&\leqslant&\sum_{q=1}^{+\infty}\sum_{t=1}^{+\infty}e^{2qsC}
\Bigg(\sum_{\omega\in\Omega_{t}}e^{-sl(h)}\Bigg)^{q}
\Bigg(\sum_{p^{r}\in P_{t}}e^{-s(1-\delta)}\Bigg)^{q}\\
&\leqslant&\sum_{q=1}^{+\infty}\sum_{t=1}^{+\infty}e^{2qsC}\times|\Omega_{t}|^{q}e^{-sql(h)}
\times|(P_{t})|^{q}e^{-sq(1-\delta)}
\end{align*}
where $ |B| $ designs the cardinal of a subset $ B $ of the group $ \Gamma $. The last two inequalities are obtained from $ l(h)\leqslant d(0,\omega.0) $, for every $ \omega\in\Omega_{t} $ and from $ t(1-\delta)\leqslant d(0,p^{r}.0) $, for every $ p^{r}\in P_{t} $ since 
$ \sum_{i=1}^{q}d(0,\omega_{i}.0)\leqslant\delta d(0,\gamma.0) $.\\
Now let be $ \varepsilon>0 $ and put $ s_{0}=\frac{1}{2}+\varepsilon $. Since there exist constants $ C_{3} $ and $ C_{4} $ such that $ |\Omega_{t}|\leqslant C_{3}e^{\delta(t+1)\delta_{\Gamma}} $ and 
$ |P_{t}|\leqslant C_{4}e^{\frac{1}{2}\delta(t+1)} $ we have so,
$$ P_{\Gamma_{\alpha}^{\delta}}(s_{0})\leqslant C_{3}C_{4}\sum_{q=1}^{+\infty}e^{k\beta}\sum_{t=1}^{+\infty}
e^{-tk\lambda} $$
where $ \lambda=s_{0}(1-\delta)-\delta(\delta_{\Gamma}+\frac{1}{2}) $ and
$ \beta=\delta(\delta_{\Gamma}+\frac{1}{2})-s_{0}(l(h)-2C) $. One verifies that $ \lambda>0 $. So
$$ \sum_{t=1}^{+\infty}e^{-tk\lambda}=\frac{e^{-k\lambda}}{1-e^{-k\lambda}}. $$
And this implies that 
$$ P_{\Gamma_{\alpha}^{\delta}}(s_{0})\leqslant C_{3}C_{4}\sum_{q=1}^{+\infty}
\frac{e^{-q(\lambda-\beta)}}{1-e^{-q\lambda}}. $$
We have $ \lambda-\beta\geq s_{0}(1+l(h)-2C)-4\delta $. So for $ \delta $ enough small if $ s_{0}(1+l(h)-c)>0 $ then $ \lambda-\beta>0 $. In this case we put 
$ U_{q}=\frac{e^{-q(\lambda-\beta)}}{1-e^{-q\lambda}} $ and verify that 
$ \frac{U_{q+1}}{U_{q}}\leqslant e^{-(\lambda-\beta)}<1 $. Alembert's ruler implies that the series
$ \sum_{q=1}^{+\infty}U_{q}$ converges. Thus for every $ \delta>0 $,
$$ \mathcal{H}^{s_{0}}_{\delta}(\Lambda_{ur}^{\infty})\leqslant C_{3}C_{4}(2C_{2})^{s_{0}}
\sum_{q=1}^{+\infty}U_{q}<\infty. $$
Since in the above inequality the member of the left is independent of $ \delta $, we have 
This implies that $ \mathcal{H}^{s_{0}}(\Lambda_{ur}^{\infty})<0$. And hence $ dim_{H}(\Lambda_{ur}^{\infty})\leqslant s_{0}=\frac{1}{2}+\varepsilon$, for every $ \varepsilon>0 $. Therefore
\begin{equation}\label{upper}
dim_{H}(\Lambda_{ur}^{\infty})\leqslant\frac{1}{2}.
\end{equation}
\subsection{The lower bound}
A good way to obtain a lower-bound passes by using the following lemma, called \textbf{Frostman's Lemma}.
\begin{lem}
	Let $ F $ be a Borel set of $ \mathbb{S}^{1} $. Let $ \mu $ be a positive measure on $ F $ and assume that for some $ s>0 $ there exist constant $ C>0 $ such that for every $ \xi\in\mathbb{S}^{1} $ and $ r>0 $, we have
	$$ \mu(B(\xi,r))\leqslant Cr^{s}. $$
	Then $ \mathcal{H}^{s}(F)\geqslant\frac{\mu(F)}{c} $ and so $$ s\leqslant dim_{H}(F). $$
\end{lem}
We construct in three steps a subset $ E $ of $ \Lambda_{ur}^{\infty} $ on which we apply Frostman's Lemma.\\ 
\textbf{\underline{Step 1: Choose of sub-alphabet}}\\\\ 
Consider the alphabet $ \mathcal{A}_{1}=\{h^{\pm1}, p^{\pm1}\} $. Let 
$ S^{+}=\{s=(s_{n})_{n\geq1}/s_{n}\in\mathcal{A}_{1}, s_{n+1}\neq s_{n}^{-1}\;
\mbox{and if }\;s_{n}=p^{\pm1},\exists\;m>n\;\mbox{with}\;s_{m}\neq s_{n}\} $. The map which sends every sequence $ s=(s_{n})_{n\geq1}\in S^{+} $ to the point $ \xi\in\Lambda_{r} $ defined by $ \xi=\lim_{n\rightarrow+\infty}\gamma_{n}^{\xi}(0) $ where $ \gamma_{n}^{\xi}=s_{1}s_{2}...s_{n} $, is into one-to-one correspondence. We put $ lenght(\gamma_{n}^{\xi})=l(\gamma_{n}^{\xi})=n $ in the alphabet $ \mathcal{A}_{1} $.\\
Consider the subset $ \Lambda_{ur}^{\infty,+} $ consisting of elements of $ \Lambda_{ur}^{\infty}\subset\Lambda_{r} $ such that any term of their unique associated sequence in $ S $ does not equal to $ p^{-1} $. In other words the elements of $ S^{+} $ associated to the ones of 
$ \Lambda_{ur}^{\infty,+} $ are written in the alphabet $ \mathcal{A}_{1}^{+}=\{h^{\pm1}, p\} $.\\
Even it means choosing the positive integer $ n $ larger enough we rewrite for every $ \xi\in\Lambda_{ur}^{\infty,+} $ the isometry $ \gamma_{n}^{\xi} $ in the alphabet 
$ \mathcal{A}_{2}=\{h^{\mathbb{Z}^{\star}}, p^{\mathbb{N}}\} $:
$$ \gamma_{n}^{\xi}=\gamma_{l}^{\xi}=c_{1}...c_{l}. $$ 
Define $$ \Gamma_{E}:=\{\gamma=c_{1}...c_{l}\;/\;c_{i}\in\mathcal{A}_{2}\} $$ put 
$$ A_{n}=\{z\in\mathbb{D}\;/\;n\leqslant d(0,z)<n+1\}, $$ for every $ n\geq1 $ and set $$ q=\min d(s_{i}.0,s_{j}.0) $$ where $ s_{i} $ and $ s_{j} $ belong in 
$ \mathcal{A}_{1}=\{h^{\pm1}, p^{\pm1}\} $. We claim the following two lemmas
\begin{lem}\label{disjoint}
According to the notations of Lemma \ref{kaim} (see Figure \ref{dibujo11}) there exists a constant $ C_{5} $ such that for every two distinct points $ z $ and $ z' $ of $ A_{n}\cap\Gamma_{E}.0 $ verifying $ d(z,z')\geq q $, the balls $ B(\xi_{0,z},C_{5}e^{-d(0,z)}) $ and $ B(\xi_{0,z'},C_{5}e^{-d(0,z')}) $ are disjoint.
\end{lem}
\begin{proof}
Since $ 0<q<d(z,z') $, then $ \theta=\widehat{(z,0,z')}\neq0 $. Design $ d $ the length of arc with endpoints $ \xi_{0,z} $ and $ \xi_{0,z'} $. We have $d=2\sin\Big(\frac{\theta}{2}\Big)$. It suffices to set $$ C_{5}=\frac{1}{3}d\times e^{\min\{d(0,z),d(0,z')\}}. $$ 
\end{proof}

\begin{figure}[htbp]
	\centering
	\begin{normalsize} 
		\def\svgwidth{\textwidth}
\begingroup%
  \makeatletter%
  \providecommand\color[2][]{%
    \errmessage{(Inkscape) Color is used for the text in Inkscape, but the package 'color.sty' is not loaded}%
    \renewcommand\color[2][]{}%
  }%
  \providecommand\transparent[1]{%
    \errmessage{(Inkscape) Transparency is used (non-zero) for the text in Inkscape, but the package 'transparent.sty' is not loaded}%
    \renewcommand\transparent[1]{}%
  }%
  \providecommand\rotatebox[2]{#2}%
  \newcommand*\fsize{\dimexpr\f@size pt\relax}%
  \newcommand*\lineheight[1]{\fontsize{\fsize}{#1\fsize}\selectfont}%
  \ifx\svgwidth\undefined%
    \setlength{\unitlength}{891bp}%
    \ifx\svgscale\undefined%
      \relax%
    \else%
      \setlength{\unitlength}{\unitlength * \real{\svgscale}}%
    \fi%
  \else%
    \setlength{\unitlength}{\svgwidth}%
  \fi%
  \global\let\svgwidth\undefined%
  \global\let\svgscale\undefined%
  \makeatother%
  \begin{picture}(1,0.54208754)%
    \lineheight{1}%
    \setlength\tabcolsep{0pt}%
    \put(0,0){\includegraphics[width=\unitlength,page=1]{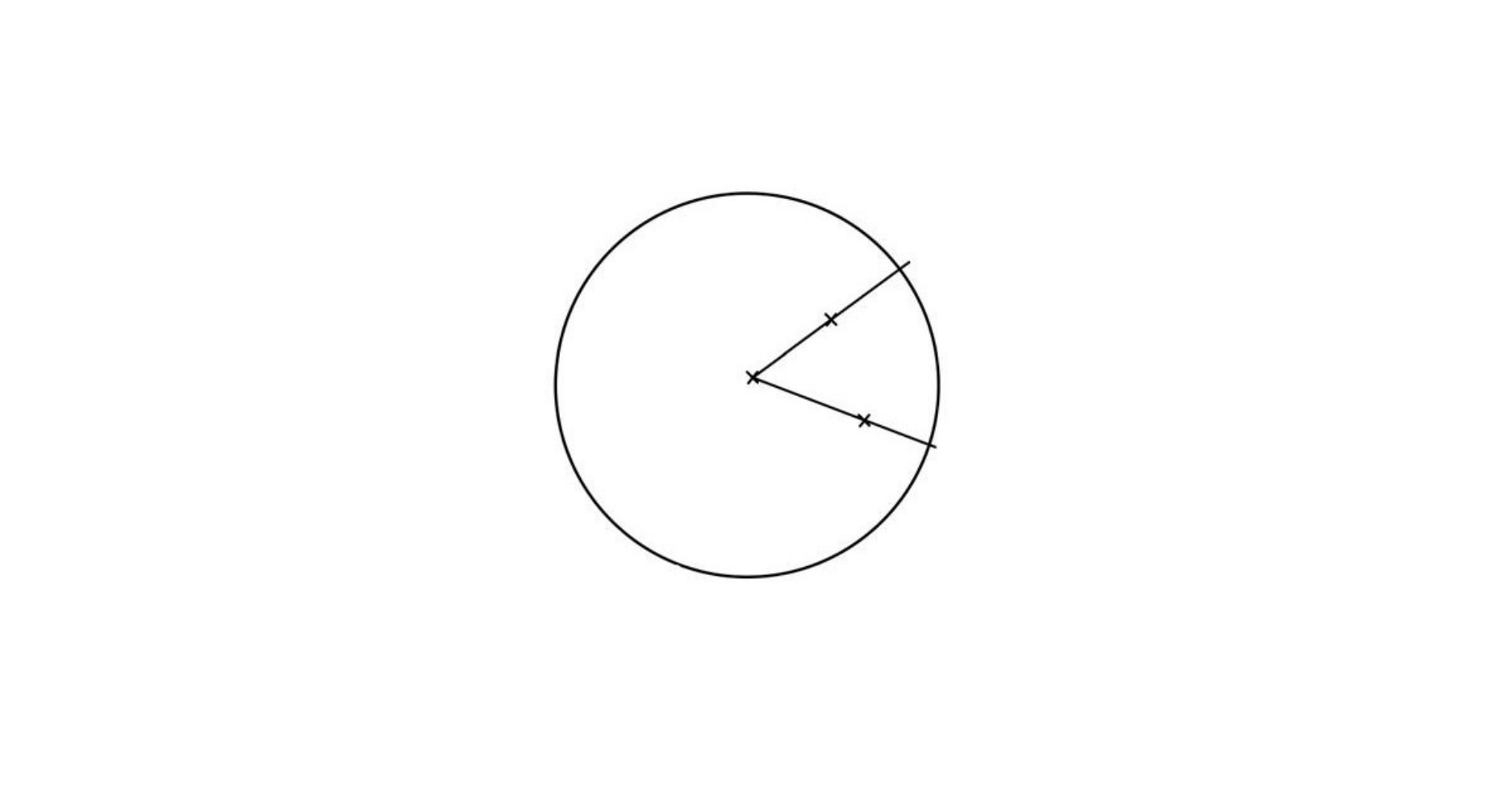}}%
    \put(0.63280115,0.36139088){\color[rgb]{0.05490196,0.05490196,0.05490196}\makebox(0,0)[lt]{\lineheight{1.25}\smash{\begin{tabular}[t]{l}$\xi_{0,z}$\end{tabular}}}}%
    \put(0.63497822,0.23729798){\color[rgb]{0.05490196,0.05490196,0.05490196}\makebox(0,0)[lt]{\lineheight{1.25}\smash{\begin{tabular}[t]{l}$\xi_{0,z'}$\end{tabular}}}}%
    \put(0.5471698,0.3396202){\color[rgb]{0.05490196,0.05490196,0.05490196}\makebox(0,0)[lt]{\lineheight{1.25}\smash{\begin{tabular}[t]{l}$z$\end{tabular}}}}%
    \put(0.55805517,0.22858974){\color[rgb]{0.05490196,0.05490196,0.05490196}\makebox(0,0)[lt]{\lineheight{1.25}\smash{\begin{tabular}[t]{l}$z'$\end{tabular}}}}%
    \put(0.47387521,0.29390177){\color[rgb]{0.05490196,0.05490196,0.05490196}\makebox(0,0)[lt]{\lineheight{1.25}\smash{\begin{tabular}[t]{l}$0$\end{tabular}}}}%
  \end{picture}%
\endgroup%

	\end{normalsize}
	\caption{Notation for $\xi_{0,z}$ and $\xi_{0,z'}$.}
	\label{dibujo11}
\end{figure}
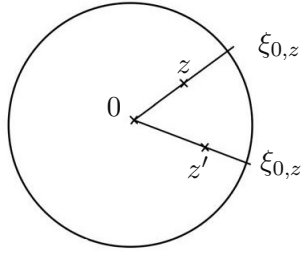

\begin{lem}\label{diverge}
For every $ \varepsilon>0 $, we have
$$ \limsup_{n\rightarrow+\infty}e^{-n(\frac{1}{2}-\varepsilon)}card\{\gamma\in\Gamma_{E}\;/\;
\gamma.0\in A_{n}\}=\infty.  $$
\end{lem}
\begin{proof}
	
Put $ s=\frac{1}{2}-\varepsilon $. For every positive integer $ n $,  $ <p^{n}> $ is a subgroup of $ \Gamma_{E} $ and that
	\begin{equation}\label{eqlem7}
		\sum_{\gamma\in\Gamma_{E}}e^{-(s+\varepsilon)d(0,\gamma.0)}=+\infty.
	\end{equation}
	If there was existed a constant $ A>0 $ such that
	$$ \sum_{\gamma\in\Gamma_{E},\gamma.0\in A_{n}}e^{-sd(0,\gamma.0)}\leqslant A $$ then we have
	$$ \sum_{\gamma\in\Gamma_{E}}e^{-(s+\varepsilon)d(0,\gamma.0)} =\sum_{n\geqslant0}\sum_{\gamma\in\Gamma_{E},\gamma.0\in A_{n}}e^{-(s+\varepsilon)d(0,\gamma.0)}$$
	$$=\sum_{n\geqslant0}\Big[\Big(\sum_{\gamma\in\Gamma_{E},\gamma.0\in A_{n}}e^{-sd(0,\gamma.0)}\Big)
	\Big(\sum_{\gamma\in\Gamma_{E},\gamma.0\in A_{n}}e^{-\varepsilon d(0,\gamma.0)}\Big)\Big]
	\leqslant A\sum_{n\geqslant0}e^{-\varepsilon n}<\infty$$ which is in contradiction with the equation \ref{eqlem7}. Then 
	$$ \limsup_{n\rightarrow+\infty}\sum_{\gamma\in\Gamma_{E},\gamma.0\in A_{n}}e^{-sd(0,\gamma.0)}=+\infty. $$ Hence
	$$ \limsup_{n\rightarrow+\infty}e^{-n(\frac{1}{2}-\varepsilon)}card\{\gamma\in\Gamma_{E}\;/\;
	\gamma.0\in A_{n}\}=\infty.  $$
\end{proof}

\textbf{\underline{Step 2: Graph and tree}}\\\\
For every $ \xi=s_{1}...s_{i}...=c_{1}...c_{l}...\in\Lambda_{ur}^{\infty,+} $ we set the following definitions.
\begin{defn}
\begin{enumerate}
\item \textbf{Vertex:} For every $ n\geq0 $, the point 
$ \gamma_{n}^{\xi}(0)=c_{1}...c_{n}(0)\in\mathbb{D} $ where $ c_{i}\in\mathcal{A}_{2}^{+} $ is called \textbf{a vertex at level $ n $ relative to $ \xi $}.
\item \textbf{Edge:} The geodesic segment $ \Big[\gamma_{n-1}^{\xi}(0),\gamma_{n}^{\xi}(0)\Big] $ is called \textbf{an edge relative to $ \xi $}.
\item \textbf{Path:} The infinite union 
$ \varUpsilon^{\xi}:=\bigcup_{n\geq1}\Big[\gamma_{n-1}^{\xi}(0),\gamma_{n}^{\xi}(0)\Big] $ where 
$ \gamma_{0}^{\xi}(0)=0\in\mathbb{D} $ is called \textbf{a path relative to $ \xi\in\Lambda_{ur}^{\infty,+} $}.
\item \textbf{Path crossing :} Let $ \xi $ and $ \xi' $ be two distinct points of $ \Lambda_{ur}^{\infty,+} $. The paths $ \varUpsilon^{\xi} $ and $ \varUpsilon^{\xi'} $ intersect if there exists a positive integer $ n>0 $ such that the vertices $ \gamma_{n}^{\xi}(0) $ and $ \gamma_{n}^{\xi'}(0) $ coincide. Thus intersection vertices is called a \textbf{path-crossing}.
\end{enumerate}
\end{defn}
The union $ \mathcal{G}=\bigcup_{\xi\in\Lambda_{ur}^{\infty,+}}\varUpsilon^{\xi} $ is by construction a graph which vertices are the points $ \Big(\gamma_{n}^{\xi}(0)\Big)_{n\geq0,\;\xi\in\Lambda_{ur}^{\infty,+}} $, see Figure \ref{graph}.


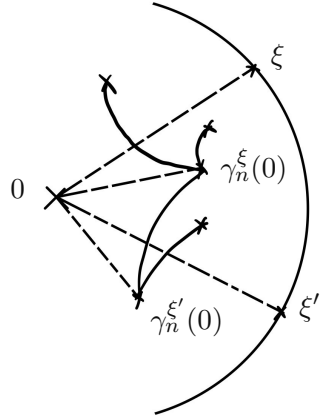
\begin{figure}[htbp]
	\centering
	\begin{normalsize} 
		\def\svgwidth{10cm}
\begingroup%
  \makeatletter%
  \providecommand\color[2][]{%
    \errmessage{(Inkscape) Color is used for the text in Inkscape, but the package 'color.sty' is not loaded}%
    \renewcommand\color[2][]{}%
  }%
  \providecommand\transparent[1]{%
    \errmessage{(Inkscape) Transparency is used (non-zero) for the text in Inkscape, but the package 'transparent.sty' is not loaded}%
    \renewcommand\transparent[1]{}%
  }%
  \providecommand\rotatebox[2]{#2}%
  \newcommand*\fsize{\dimexpr\f@size pt\relax}%
  \newcommand*\lineheight[1]{\fontsize{\fsize}{#1\fsize}\selectfont}%
  \ifx\svgwidth\undefined%
    \setlength{\unitlength}{1995bp}%
    \ifx\svgscale\undefined%
      \relax%
    \else%
      \setlength{\unitlength}{\unitlength * \real{\svgscale}}%
    \fi%
  \else%
    \setlength{\unitlength}{\svgwidth}%
  \fi%
  \global\let\svgwidth\undefined%
  \global\let\svgscale\undefined%
  \makeatother%
  \begin{picture}(1,0.79548872)%
    \lineheight{1}%
    \setlength\tabcolsep{0pt}%
    \put(0,0){\includegraphics[width=\unitlength,page=1]{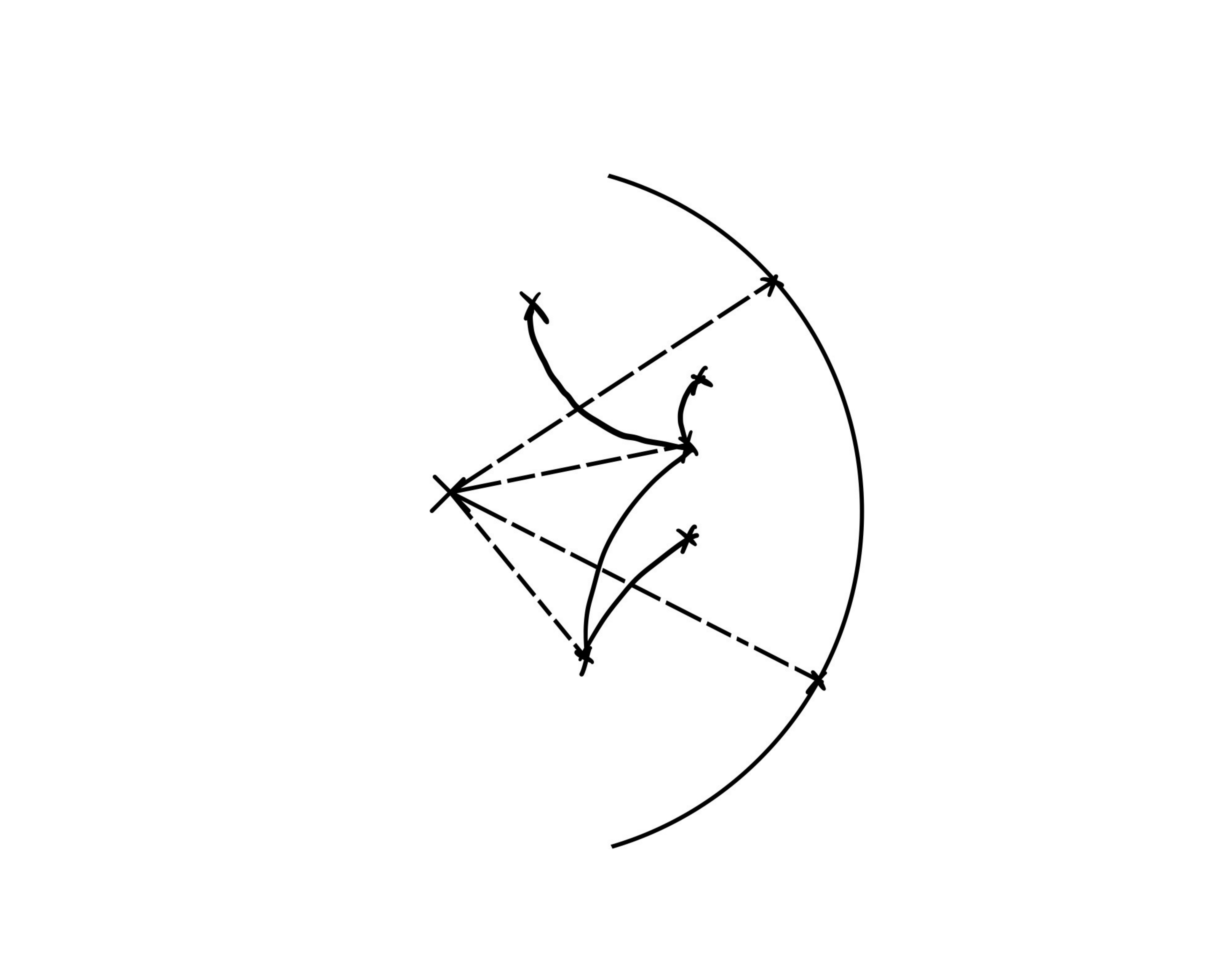}}%
    \put(0.30424295,0.39532205){\color[rgb]{0.05490196,0.05490196,0.05490196}\makebox(0,0)[lt]{\lineheight{1.25}\smash{\begin{tabular}[t]{l}$0$\end{tabular}}}}%
    \put(0.65014974,0.5726356){\color[rgb]{0.05490196,0.05490196,0.05490196}\makebox(0,0)[lt]{\lineheight{1.25}\smash{\begin{tabular}[t]{l}$\xi$\end{tabular}}}}%
    \put(0.68793785,0.23835596){\color[rgb]{0.05490196,0.05490196,0.05490196}\makebox(0,0)[lt]{\lineheight{1.25}\smash{\begin{tabular}[t]{l}$\xi'$\end{tabular}}}}%
    \put(0.58329381,0.42245199){\color[rgb]{0.05490196,0.05490196,0.05490196}\makebox(0,0)[lt]{\lineheight{1.25}\smash{\begin{tabular}[t]{l}$\gamma_n^\xi(0)$\end{tabular}}}}%
    \put(0.49027682,0.22091528){\color[rgb]{0.05490196,0.05490196,0.05490196}\makebox(0,0)[lt]{\lineheight{1.25}\smash{\begin{tabular}[t]{l}$\gamma_n^{\xi'}(0)$\end{tabular}}}}%
  \end{picture}%
\endgroup%

	\end{normalsize}
	\caption{The graph $\mathcal{G}$} 
	\label{graph}
\end{figure}

By construction, the boundary at infinity of the graph $ \mathcal{G} $ (its accumulation points) is exactly the set $ \Lambda_{ur}^{\infty,+} $.\\
All paths of $ \mathcal{G} $ go from $ \gamma_{0}^{\xi}(0)=0\in\mathbb{D} $, for every $\xi\in\Lambda_{ur}^{\infty,+} $ . So the point $ 0\in\mathbb{D} $ is the \textit{root} of the graph $ \mathcal{G} $.\\
On $ \mathcal{G} $ each vertex $ z $ has one or many finite "parents" and many "children". We denote
$ \mathcal{G}(z)$, the set of children of a vertex $ z $. And the set of vertices at level $ n $, denoted $ \mathcal{G}_{n} $ is the set of points $ \gamma(0) $ such that $ \gamma\in\Gamma_{E} $ and 
$ l(\gamma)=n $ with respect the sub-alphabet $ \mathcal{A}_{2}^{+} $.
\begin{defn}(Subpaths)
\begin{enumerate}
\item \textbf{Previous subpath :} The union 
$ \varUpsilon_{-}^{\xi}:=\bigcup_{k=1}^{n}\Big[\gamma_{k-1}^{\xi}(0),\gamma_{k}^{\xi}(0)\Big] $ is called the \textbf{previous subpath} of the vertex $ \gamma_{n}^{\xi}(0) $ relative to $ \xi $.
\item \textbf{Coming subpath :} The union 
$ \varUpsilon_{+}^{\xi}:=\bigcup_{k\geq n+1}\Big[\gamma_{k-1}^{\xi}(0),\gamma_{k}^{\xi}(0)\Big]$ is called the \textbf{coming subpath} of the vertex $ \gamma_{n}^{\xi}(0) $ relative to $ \xi $.
\end{enumerate}
\end{defn}

Now we are ready to draw a tree $ \mathcal{T} $ from the graph $ \mathcal{G} $ in the following way: At every path-crossing we keep only one previous subpath and the rest of the previous branches of the corresponding vertex are deleted, see Figure \ref{tree}. In addition, we will always choose the branch which contains the highest number of vertex.
So every vertex of $\mathcal G$ now has a single parent. Hence the graph $ \mathcal{G} $ becomes a tree $ \mathcal{T} $ which boundary at infinity is $ E\subseteq\Lambda_{ur}^{\infty,+}\subset\Lambda_{ur}^{\infty} $. We design respectively
$ \mathcal{T}(z) $:= the set of "children" belonging to $ \mathcal{T} $ of a vertex $ z $ and $ \mathcal{T}_{n} $:= the set of vertices at level $ n $ belonging to $ \mathcal{T} $. We have $ \mathcal{T}(z)\subset\mathcal{G}(z) $ and $ \mathcal{T}_{n}\subset\mathcal{G}_{n} $.\\


\begin{figure}[htbp]
	\centering
	\begin{normalsize} 
		\def\svgwidth{\textwidth}
\begingroup%
  \makeatletter%
  \providecommand\color[2][]{%
    \errmessage{(Inkscape) Color is used for the text in Inkscape, but the package 'color.sty' is not loaded}%
    \renewcommand\color[2][]{}%
  }%
  \providecommand\transparent[1]{%
    \errmessage{(Inkscape) Transparency is used (non-zero) for the text in Inkscape, but the package 'transparent.sty' is not loaded}%
    \renewcommand\transparent[1]{}%
  }%
  \providecommand\rotatebox[2]{#2}%
  \newcommand*\fsize{\dimexpr\f@size pt\relax}%
  \newcommand*\lineheight[1]{\fontsize{\fsize}{#1\fsize}\selectfont}%
  \ifx\svgwidth\undefined%
    \setlength{\unitlength}{1440bp}%
    \ifx\svgscale\undefined%
      \relax%
    \else%
      \setlength{\unitlength}{\unitlength * \real{\svgscale}}%
    \fi%
  \else%
    \setlength{\unitlength}{\svgwidth}%
  \fi%
  \global\let\svgwidth\undefined%
  \global\let\svgscale\undefined%
  \makeatother%
  \begin{picture}(1,0.51770833)%
    \lineheight{1}%
    \setlength\tabcolsep{0pt}%
    \put(0,0){\includegraphics[width=\unitlength,page=1]{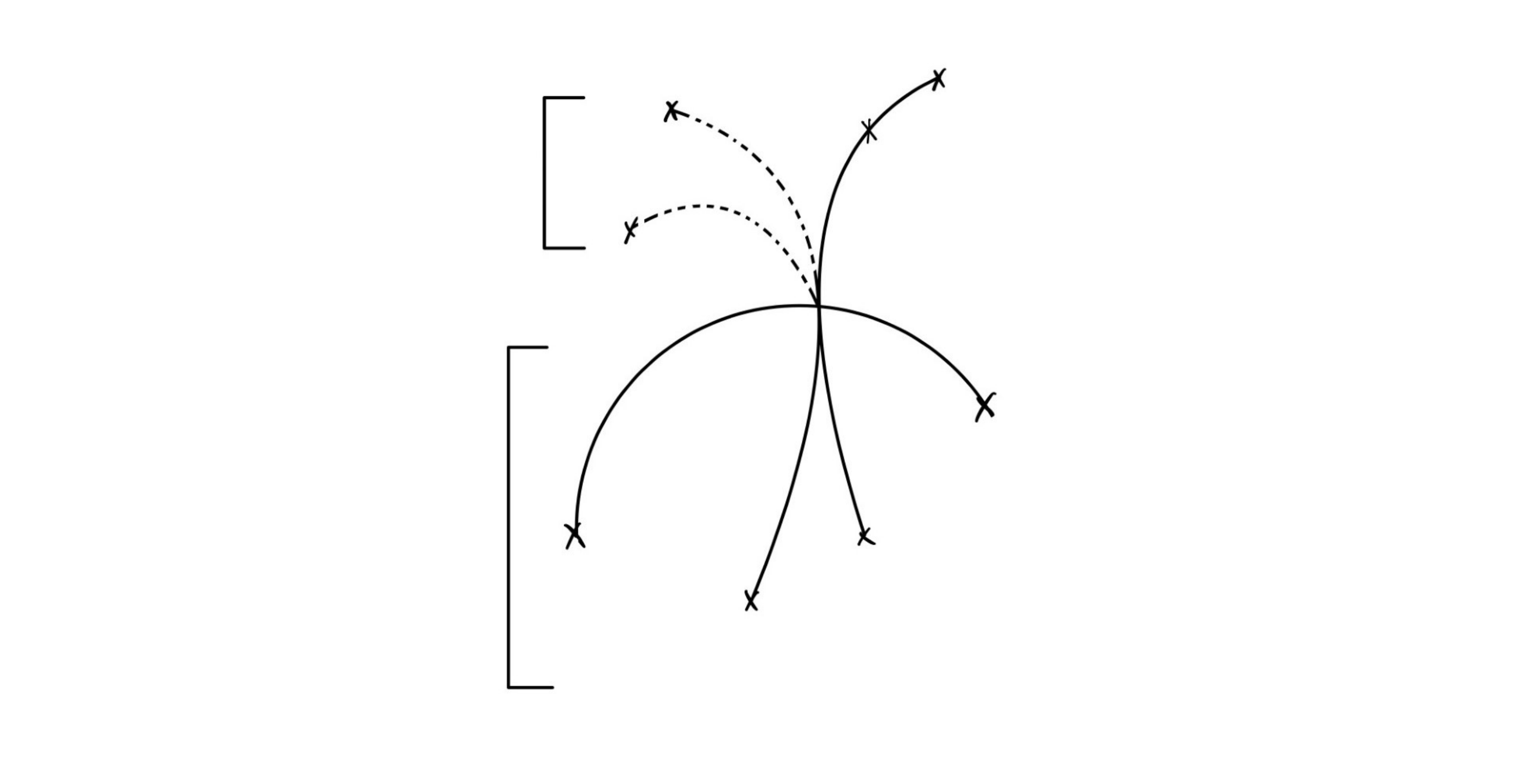}}%
    \put(0.2467344,0.41248337){\color[rgb]{0.05490196,0.05490196,0.05490196}\makebox(0,0)[lt]{\lineheight{1.25}\smash{\begin{tabular}[t]{l}Deleted\\parents\end{tabular}}}}%
    \put(0.21770682,0.20348482){\color[rgb]{0.05490196,0.05490196,0.05490196}\makebox(0,0)[lt]{\lineheight{1.25}\smash{\begin{tabular}[t]{l}Children\end{tabular}}}}%
    \put(0.56023223,0.3283034){\color[rgb]{0.05490196,0.05490196,0.05490196}\makebox(0,0)[lt]{\lineheight{1.25}\smash{\begin{tabular}[t]{l}$\gamma_{n}^{\xi}(0)$\end{tabular}}}}%
  \end{picture}%
\endgroup%

	\end{normalsize}
	\caption{The tree $ \mathcal{T} $}
	\label{tree}
\end{figure}
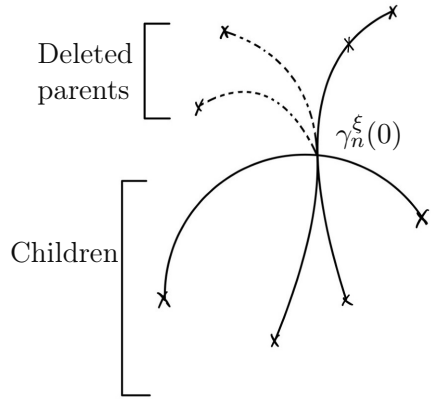

\textbf{\underline{Step 3: A measure on $ E $}}\\\\
We follow steps by steps the construction of Bishop-Jones in \cite{bijo}. Let $ z \in A_n\cap \Gamma_E.0,$  and in order to simplify  notation we put $ B_{z}=B(\xi_{0,z},C_{3}e^{-d(0,z)}) $. Define $ E_{n}:=\bigcup_{z\in\mathcal{T}_{n}}B_{z} $. Thus $ E=\bigcap_{n\geq0}E_{n} $. We define a probability measure $ \mu $ supported on $ E $ by setting $ \mu(E_{0})=1 $ and for $ z\in\mathcal{T} $ and $ z'\in\mathcal{T}(z) $ set $$ \mu(B_{z'})=
\frac{e^{-(\frac{1}{2}-\varepsilon)d(0,z')}}{\sum_{w\in\mathcal{T}(z)}
e^{-(\frac{1}{2}-\varepsilon)d(0,w)}}\mu(B_{z}). $$
For every $ z\in\mathcal{T} $, by the Lemma \ref{diverge} and induction we have 
$$ \mu(B_{z})\leqslant C_{3}e^{-(\frac{1}{2}-\varepsilon)d(0,z)}. $$
From Lemma  \ref{disjoint}, if $ z $ and $ z' $ are distinct points on $ \mathcal{T}_{n} $ then $ B_{z}\cap B_{z'}=\emptyset $.\\
Now let $ B $ be a ball of $ \mathbb{S}^{1} $. If $ B\cap E=\emptyset $, then $ \mu(B)=0 $ since the measure $ \mu $ is supported on $ E $. If there exists $ \eta\in B\cap E $ then there exists a sequence $ (z_{n}) $ of elements of the tree $ \mathcal{T} $ which tends to $ \eta $. Design $ B_{z_{0}} $ the lowest generation disc in our construction such that $ B_{z_{0}}\cap B\neq\emptyset $ but $ B\nsubseteq2B_{z_{0}} $. Let $ B_{z_{1}} $ be the "parent" of $ B_{z_{0}} $. By the maximality of $ B_{z_{0}} $ we have $ B\subset2B_{z_{1}} $. Since $ 2B_{z_{1}} $ is disjoint from any other balls of the same generation
$$ \mu(B)\leqslant\mu(B_{z_{1}})\leqslant
C_{5}e^{-(\frac{1}{2}-\varepsilon)d(0,z_{1})}. $$
To continuous let us establish the following lemma.
\begin{lem}
There exists a constant $ C_{6} $ such that for every vertex $ z\in\mathcal{T} $ and every $ z'\in\mathcal{T}(z) $ we have $ r_{z}\leqslant C_{6}r_{z'} $.
\end{lem}
\begin{proof}
We have $ \frac{r(B_{z})}{r(B_{z'})}=e^{d(z,z')} $. The points $ z $ and $ z' $ belong on two consecutive annulus $ A_{n}\cup A_{n+1} $, for some integer $ n $ or there exists $ d>0 $ such that for every positive integer $ n $, $ diam(A_{n}\cup A_{n+1})\leqslant d $. So $ \frac{r_{z}}{r_{z'}}\leqslant e^{d} $. Hence it is enough to take $ C_{6}=e^{d} $.	
\end{proof}
From the above lemma and the fact $ B\nsubseteq2B_{z_{0}} $ we have
$$ \mu(B)\leqslant C_{5}(2C_{6})^{\frac{1}{2}-\varepsilon}e^{-(\frac{1}{2}-\varepsilon)d(0,z_{0})}
\leqslant C_{5}(2C_{6})^{\frac{1}{2}-\varepsilon}\times r^{-(\frac{1}{2}-\varepsilon)d(0,z)} $$
where $ r=r(B) $. So from Frostman's Lemma for every $ \varepsilon>0 $, 
$$ \frac{1}{2}-\varepsilon\leqslant dim_{H}(E)\leqslant dim_{H}(\Lambda_{ur}^{\infty,+})\leqslant dim_{H}(\Lambda_{ur}^{\infty}). $$ 
And therefore 
\begin{equation}\label{lower}
dim_{H}(\Lambda_{ur}^{\infty})\geqslant\frac{1}{2}. 
\end{equation}	
Finally the inequalities \ref{upper} and \ref{lower} imply that 
$$ dim_{H}(\Lambda_{ur}^{\infty})=\frac{1}{2}. $$
	
Hence since $ \Lambda_{\infty}=\Lambda_{ur}^{\infty}\sqcup\Lambda_{p} $ and $ dim_{H}(\Lambda_{p})=0 $ we have $$ dim_{H}(\Lambda_{\infty})=\frac{1}{2}. $$


\begin{thebibliography}{99}	
	
	\bibitem{basm}Basmajian, A. \textit{Universal legnth bounds for non-simple closed geodesics on hyperbolic surfaces}, Journal of Topology (2013), 12 pages. London Mathematical Society doi:10.1112/jtopol/jtt005.
	
	\bibitem{Bourdon} Bourdon, M. (1995). Structure conforme au bord et flot geodesic d'un CAT (-1)-espace. \textit{L'Enseignement Math}, 41, 63-102.
	
	\bibitem{bijo} Bishop, C. J. and Jones, P. W. \textit{Hausdorff dimension and Kleinian groups}, Acta Math., 179 (1997), 1-39.
	
	\bibitem{BridsonHaefliger}Bridson, M. R. and Haefliger, A. \textit{Metric spaces of non-positive curvature} Springer 1964.
	
	\bibitem{Buser} Buser, P. (2010). \emph{Geometry and spectra of compact Riemann surfaces}. Springer Science and Business Media.
	
	\bibitem{dalbo}Dal'Bo, F. \textit{Geodesic and horoccylic trajectories}, Springer 2011.
	
	\bibitem{Dalbo-Starkov}Dal'bo, F. and Starkov, A. N. (2000). On a classification of limit points of infinitely generated Schottky groups. \emph{Journal of dynamical and control systems}, 6(4), 561-578.
	
	\bibitem{dani}Dani, S. G. (1985). \emph{Divergent trajectories of flows on homogeneous spaces and Diophantine approximation}.
	
	\bibitem{felipe}Riquelme, F. and Velozo, A. \textit{On the hausdorff dimension of geodesics that escape on average}.\emph{ arXiv preprint} arXiv:2308.05894.
	
	
	
	\bibitem{matila}Mattila, P. \textit{Fourier Analysis and Hausdorff dimension}. Cambridge University Press 2015.
	
	
	
	
	\bibitem{Vo} Vo, H. (2022). Short closed geodesics on cusped hyperbolic surfaces. \emph{Pacific Journal of Mathematics}, 318(1), 127-151.
	
	\bibitem{weiss}Weiss, B., \textit{Divergent trajectories on noncompact parameter spaces} GAFA, vol. 14 (2004) 94-149. 1016-443X/04/010094-56 DOI 10.1007/s00039-004-0453-z.
	
\end{thebibliography}
\end{document}